 \documentclass[11pt]{article}
 \UseRawInputEncoding
\usepackage{amssymb, amsthm, amsmath, amscd}
\usepackage{tikz-cd}

\newtheorem{thm}{Theorem}[section]
\newtheorem{lem}[thm]{Lemma}
\newtheorem{prop}[thm]{Proposition}
\newtheorem{cor}[thm]{Corollary}
\newtheorem{NN}[thm]{}

\theoremstyle{definition}\newtheorem{df}[thm]{Definition}
\theoremstyle{definition}\newtheorem{rem}[thm]{Remark}
\theoremstyle{definition}

\renewcommand{\phi}{\varphi}

\newcommand{\N}{\mathbb{N}}

\newcommand{\Q}{\mathbb{Q}}

\newcommand{\C}{\mathbb{C}}

\newcommand{\Aff}{\operatorname{Aff}}

\newcommand{\id}{\operatorname{id}}

\newcommand{\hm}{homomorphism}
\newcommand{\dt}{\delta}
\newcommand{\ep}{\epsilon}
\newcommand{\la}{\langle}
\newcommand{\ra}{\rangle}
\newcommand{\andeqn}{\,\,\,{\rm and}\,\,\,}
\newcommand{\rforal}{\,\,\,{\rm for\,\,\,all}\,\,\,}
\newcommand{\CA}{$C^*$-algebra}
\newcommand{\SCA}{$C^*$-subalgebra}

\newcommand{\af}{{\alpha}}
\newcommand{\bt}{{\beta}}

\newcommand{\beq}{\begin{eqnarray}}
\newcommand{\eneq}{\end{eqnarray}}
\newcommand{\tforal}{\,\,\,\text{for\,\,\,all}\,\,\,}
\newcommand{\tand}{\,\,\,\text{and}\,\,\,}
\newcommand{\Qw}{\overline{QT(A)}^w}
\newcommand{\LAff}{{\rm LAff}}
\newcommand{\cuapprox}{\stackrel{\approx}{\sim}}
\newcommand{\Wlog}{Without loss of generality}

\newcommand{\simle}{\stackrel{\sim}{<}}

\title{
Projective Hilbert modules and sequential approximations}
\author{Lawrence G. Brown and Huaxin Lin}

\begin{document}

\maketitle

\begin{abstract}
We show  that, when $A$ is a separable 
\CA, every countably generated Hilbert $A$-module is projective (with 
bounded module maps as morphisms).
We also study the approximate extensions of bounded module maps.
In the case that $A$ is a $\sigma$-unital simple \CA\, with strict comparison and every
strictly positive lower semicontinuous affine function on quasitraces can be realized 
as the rank of an element in Cuntz semigroup, 
we show that the Cuntz semigroup is the same as unitarily equivalent class of countably generated 
Hilbert $A$-modules if and only if $A$ has stable rank one. 

\end{abstract}

\section{Introduction}
We study Hilbert module over a \CA\, $A.$ 
A Hilbert $A$ -module $H$  is said to be projective (with bounded module maps as morphisms)
if, for any 
Hilbert $A$-modules $H_1$ and $H_2,$ any bounded module map 
$\phi: H\to H_1$ and  any surjective bounded module map $s: H_2\to H_1,$  there  is always 
a bounded module map $\psi: H\to H_2$ such that $\phi=s\circ \psi.$ 
 It is easy to see that, when $A$ is a unital \CA, a   direct sum $A^{(n)}$ of $n$ copies of $A$ is 
a projective Hilbert $A$-module.
However,  it  is not known that this holds for non-unital \CA s.
One  observes that, for a non-unital \CA\, $A,$ it itself  is 
not algebraically finitely generated as an $A$-module.
  The characterization of projective Hilbert modules over a 
\CA\,  $A$ seems to remain elusive for decades.  Evidence, on the other hand, suggests that most countably 
generated Hilbert modules are projective. In this paper, we will confirm  that, for  any separable \CA\, $A,$ 
every countably generated Hilbert $A$-module is projective. 

It is attempting to classify countably generated Hilbert $A$-modules.
It was shown in \cite{CEI} that, when $A$ has stable rank one, unitary equivalence classes of Hilbert 
$A$-modules are determined by the Cuntz semigroup of $A.$  One might prefer  a more quantified description 
of Hilbert $A$-modules using  some version of dimension functions.   Indeed,
in the  case that $A$ is a separable simple 
\CA\, such that its purely non-compact elements in the Cuntz semigroup are  determined by  
 strictly positive lower semicontinuous affine functions on its 
quasitraces, then one wishes to use the same functions to described those Hilbert $A$-modules 
which are not algebraically finitely generated.  We show that {{in this case,}}
Cuntz equivalence classes are the same as unitary equivalence classes of  countably generated Hilbert 
$A$-modules if and only if $A$ has tracial approximate oscillation zero, or equivalently, in this case, 
$A$ has stable rank one. This is a partial converse of the theorem in \cite{CEI} mentioned above.

Injective Hilbert $A$-modules were studied  in \cite{Lninj}.  If $A$ is not an $AW^*$-algebra,
then $A$ itself is not  an injective   Hilbert $A$-module  (with bounded module maps as morphisms) (see 
Theorem 3.14 of \cite{Lninj}). 
Even if we only consider bounded module maps with adjoints,  there are only very few injective Hilbert 
modules. 
It was shown in \cite{Lninj} that a  countably generated Hilbert $A$-module is *-injective if only if it is orthogonal complementary.  
Let $H$ be a countably generated Hilbert $A$-module. Suppose that $H_0\subset H_1$ are
Hilbert $A$-modules and $\phi: H_0\to H$ is a bounded module map.
As we mentioned
$\phi$ may not be extended to a bounded module map
{{from $H_1$ to $H.$}} 
However, we  show that one may find a sequence of bounded module maps 
$\phi_n: H_1\to H$ with $\|\phi_n\|\le \|\phi\|$  such that
$\lim_{n\to\infty} \|\phi_n(x)-\phi(x)\|=0$ for all $x\in H_0$ (note that {{$H_0$}} is not assumed to 
{{be}} separable).
This result may be stated as every countably generated Hilbert module 
over a $\sigma$-unital \CA\, $A$ is ``sequentially approximately injective". 
With the same spirit, we show  that every countably generated Hilbert module 
over a $\sigma$-unital \CA\, is ``sequentially approximately projective". 

The paper is organized as follows:
Section 2 collects some easy facts about projective Hilbert modules and algebraically finitely generated 
Hilbert modules over a \CA. 
Section 3  discusses some basic results about countably generated Hilbert A-modules.
In Section 4, we show that, under the assumption that $A$ is a $\sigma$-unital simple 
\CA\, with strict comparison and the canonical (dimension function) map $\Gamma$ from 
Cuntz semigroup to lower semi-continuous affine functions on quasitraces 
${\widetilde{QT}}(A)$ is surjective,  countably generated (but not algebraically generated) Hilbert $A$-modules 
can be classified by these lower semi-continuous affine functions if and only if $A$ has stable rank one.
This is a partial converse of  a theorem in \cite{CEI}. We also show that, assume that
$A$ is a $\sigma$-unital simple \CA\, with finite radius of comparison, then a countably generated Hilbert 
$A$-module with infinite quasitrace is unitarily equivalent to $l^2(A).$ 
In section 5, we show that  every countably generated Hilbert module over a separable \CA\, $A$ is always 
projective. Section 6 shows that every Hilbert $A$-module is ``approximately injective," and 
every countably generated Hilbert module over a $\sigma$-unital \CA\, $A$ is ``approximately projective".

{\bf Acknowledgement}: 
This work is based on a preprint \cite{LinCuntz} of 2010.  A draft of the current paper was formed in 2014.
The second named author was partially supported by a NSF grant (DMS-1954600). Both authors would like to acknowledge the support during their visits
to the Research Center of Operator Algebras at East China Normal University
which is partially supported by Shanghai Key Laboratory of PMMP, Science and Technology Commission of Shanghai Municipality (\#13dz2260400 and 
\#22DZ2229014).

\section{Easy facts about projective Hilbert modules}

\vspace{0.1in}

\begin{df}\label{Hild}
Let $A$ be a \CA. For an integer $n\ge 1,$ denote by $A^{(n)}$ the
(right) Hilbert $A$-module of orthogonal direct sum of $n$ copies of $A.$ If
$x=(a_1, a_2,...,a_n), y=(b_1,b_2,...,b_n),$ then
$$
\la x,y\ra =\sum_{i=1}^n a_n^*b_n.
$$
Denote by $H_A,$ or $l^2(A),$ the standard countably generated Hilbert 
$A$-module
$$
H_A=\{\{a_n\}: \sum_{n=1}^k a_n^*a_n \,\,\,{\rm converges\,\,\,
in\,\,\, norm\,\, as}\,\, k\to\infty\},
$$
where the inner product is defined by
$$
\la \{a_n\}, \{b_n\}\ra =\sum_{n=1}^{\infty} a_n^*b_n.
$$

Let $H$ be a Hilbert $A$-module. denote by $H^{\sharp}$ the set of
all  bounded $A$-module maps from $H$ to $A.$
For each $y\in H,$ define $\hat{y}(x)=\la y, x\ra$ for all $x\in H.$ Then 
$\hat{y}\in H^\sharp.$ We say $H$ is self-dual, if every $f\in H^\sharp$ has the form 
$\hat{y}$ for some $y\in H.$ 

 If $H_1, H_2$ are
Hilbert $A$-modules, denote by $B(H_1, H_2)$ the space of all
bounded module maps from $H_1$ {{to}} $H_2.$   If $T\in B(H_1, H_2),$
denote by $T^*: H_2\to H_1^{\sharp}$  the bounded module maps
defined by
$$
T^*(y)(x)=\la y,Tx\ra \tforal x\in H_1\andeqn y\in H_2.
$$
If $T^*\in B(H_2, H_1),$ one says that $T$ has an adjoint $T^*.$  Denote by
$L(H_1, H_2)$ the set of all bounded $A$-module maps in $B(H_1,
H_2)$ with  adjoints.  Let $H$ be a Hilbert $A$-module. In what
follows, denote $B(H)=B(H,H)$ and $L(H)=L(H, H).$  $B(H)$ is a
Banach algebra and $L(H)$ is a \CA.

For $x, y\in H,$ define $\theta_{x, y}\in L(H)$ by 
$\theta_{x,y}(z)=x\la y,z\ra$ for all $z\in H.$ 
Denote by $F(H)$ the linear span of those module maps with the form
$\theta_{x, y},$ where $x, y\in H.$ Denote by $K(H)$ the
closure of $F(H).$  $K(H)$ is a \CA.  It follows from a result of
Kasparov (\cite{K}) that $L(H)=M(K(H)),$  the multiplier algebra of
$K(H),$ and, by \cite{Lnbd}, $B(H)=LM(K(H)),$ the left multiplier
algebra of $K(H).$

Suppose that $H_1$ and $H_2$ are Hilbert $A$-submodules of a Hilbert $A$-module $H.$
Then $K(H_i)$ is a hereditary \SCA\, of $K(H),$ $i=1,2,$ (see Lemma 2.13  \cite{Lninj}). 
Denote by $K(H_1, H_2)$ 
the  subspace of $K(H)$ consists of module maps 
 of the form $STL,$ where $S\in K(H_2),$ $T\in B(H)$ and $L\in K(H_1),$ where 
$K(H_2)$ and $K(H_1)$ are viewed as hereditary \SCA s of $K(H).$ 
Note that $K(H_1, H_2)$ is
 the closure  of the linear span of  those 
module maps with the form $\theta_{x, y},$ where $x\in H_2$ and $y\in H_1.$

It is convenient to have an example that $H\not=H^\sharp.$ 
Let $A$ be a \CA\, with a sequence $\{d_n\}$ of mutually orthogonal positive elements
with $\|d_n\|=1.$ 
Define $f(\{a_n\})=\sum_{n=1}^\infty d_na_n$ for all $\{a_n\}\in H_A.$ 
Then $f$ is a bounded module map but 
$f\not\in  H.$

Two Hilbert $A$-modules are said to be unitarily equivalent, or isomorphic,  if there
is an invertible map $U\in B(H_1, H_2)$ such that
$$
\la U(x_1), U(x_2)\ra =\la x_1, x_2\ra \tforal x_1, x_2\in H_1.
$$
(see \cite{Pa} for further basic  information). 

\end{df}

\begin{df}\label{DHsim}
Let $A$ be a \CA\, and $H$   a Hilbert $A$-module and $H_1$ 
 the $A^{**}$-Hilbert module extension of $H$   constructed in Section 4 of \cite{Pa}, 
Denote by $H^\sim$ the self-dual Hilbert $A^{**}$-module $H_1^\sharp$ (see Section 4 of 
\cite{Pa}). Every bounded module map in $B(H, H^\sharp)$ can be uniquely extended to a bounded module map in $B(H^\sim).$  (This easily follows from the construction of $H^\sim$ and 3.6 of \cite{Pa}.
 See also 1.3 of \cite{Lnbd}.) If $H$ is self-dual, then $B(H)= L(E)$ (see 3.5 of \cite{Pa}).
 Thus  $M(K(H)) = LM(K(H)) = QM(K(H)).$  If in addition,  $A$ is a $W^*$-algebra, $B(H)$ is also a 
 $W^*$ -algebra. In particular, $B(H^\sim)$ is a $W^*$-algebra. Since all maps in $B(H, H^\sharp)$
  can be uniquely extended to  a maps in $B(H^\sim ),$  we conclude that $B(H^\sim )$ is a $W^*$ -algebra containing 
  $K(H),$  $M(K(H)),$ $LM(K(H)$  and $QM(K(H)).$
\end{df}

\begin{df}\label{Dgen}
Let $H$ be a Hilbert $A$-module and ${\cal F}\subset H$ be a subset.
We say that $H$ is generated by ${\cal F}$ (as a Hilbert $A$-module), 
if the linear combination of elements of the form $\{za: z\in {\cal F},\,\, a\in A\}$ is dense in $H.$

We say $H$ is algebraically generated by ${\cal F},$ if every element 
$x\in H$  is a linear combination of elements of the form 
$\{za: z\in {\cal F}, a\in \tilde A\}.$ Note that we write $a\in \tilde A$ instead $a\in A$ 
to include the element $z.$  $H$ is algebraically finitely generated if $H$ is algebraically 
generated by a finite subset ${\cal F}.$ 
See Corollary \ref{Rgen} for some clarification.
\end{df}

Now we turn to the projectivity of Hilbert modules.
\begin{df}
Let $A$ be a \CA\, and $H$ be a Hilbert $A$-module.
We say that $H$ is $*$-projective, if, for any Hilbert $A$-modules $H_1$ and $H_2,$ any
module map $\phi\in L(H, H_2)$ and any surjective module map
$s\in L(H_1,H_2),$ there exists a module map $\psi\in L(H, H_1)$ such that
\beq
s\circ \psi=\phi.
\eneq
Note here we assume that $s, \phi$ and $\psi$ have adjoint module maps. 
\end{df}


\begin{thm}\label{1proj}
Let $A$ be a \CA. 


{\rm (1)}  Suppose that $H$ and $H_1$ are Hilbert $A$-modules and
 $s\in L(H_1, H)$ is surjective.
 Then there is
$\psi_1\in L(H, H_1)$ such that
\beq\label{1proj-1}
s\circ \psi_1={\rm id}_{H}.
\eneq
Moreover, $\psi_1(H)$ is an orthogonal summand of $H_1$ and 
$\psi_1(H)$ is unitarily equivalent to $H.$

{\rm (2)}   Every Hilbert $A$-module is $*$-projective.

\end{thm}

\begin{proof}
For (1), let 
$s$ be as described.  We note that $s$ has  a closed range.  Define $T:
H_1\oplus H\to H_1\oplus H$ by $T(h_1\oplus h)=0\oplus s(h_1)$ for
$h_1\in H_1$ and $h\in H.$ Then $T\in L(H_1\oplus H)=M(K(H_1\oplus
H))$ (see Theorem 1.5 of \cite{Lnbd}).  It follows from Lemma 2.4 of \cite{Lninj} that
\beq\label{1proj-3}
H_1\oplus H={\rm ker}T\oplus |T|(H_1\oplus H).
\eneq
Let $T=V|T|$ be the polar decomposition in $(K(H_1\oplus H))^{**}.$
Note that the proof of Lemma 2.4 of \cite{Lninj} shows that $0$ is an
isolated point of $|T|$ or $|T|$ is invertible.  So the same holds for $(T^*T),$ and 
hence the same holds for 
$(TT^*).$  Let $S=(TT^*)^{-1},$ where the inverse is taken in
the hereditary \SCA\, $L(H)\subset L(H_1\oplus H).$ Since $s$ is
surjective,
\beq\label{1proj-4}
|TT^*|(H)=H.
\eneq
Note that $T^*=|T|V^*,$ $V|T|V^*=|TT^*|^{1/2}$ and $V^*|TT^*|^{1/2}=T^*.$
Hence 
\beq\label{1proj-5}
L_1:=V^*(TT^*)^{-1/2}=V^*(TT^*)^{1/2}S=T^*S\in L(H_1\oplus H).
\eneq
This also implies that $L_1(H)=V^*((TT^*)^{-1/2}(H))=V^*(H)$ is closed.
Moreover, $V^*$ gives a unitary equivalence of $H$ and $L_1(H).$ 
One then checks that
\beq\label{1proj-7}
TL_1= V|T|V^*(TT^*)^{-1/2}=P,
\eneq
where $P$ is the range projection of $(TT^*)^{1/2}$ which gives the
identity map on  $H.$   Since both $T$ and $L_1$ are in $L(H_1\oplus H),$
so is $P.$ 
One then defines $\psi_1=L_1P|_{H}: H\to H_1.$ Thus $s\circ
\psi_1={\rm id}_H.$

Next we note that $H\subset {\rm ker}T$ and $|T|(H_1\oplus H)=|T|(H_1)\subset H_1.$ 
It follows that ${\rm ker}T=H_0\oplus H.$ By 
\eqref{1proj-3}  again,  $H_1\oplus H=H_0\oplus H\oplus |T|(H_1).$ 
It follows that 
$|T|(H)$ is an orthogonal summand of $H_1.$ 

On the other hand,  since  $\psi_1(H)=L_1(H)=V^*(H),$  $\psi_1(H)$
is closed. But we also have $\psi_1(H)=T^*S(H)= T^*T(H_1)=|T|(H_1).$ 
As we have shown that $L_1(H)$ is unitarily equivalent to $H,$ this  proves the ``Moreover" part of (1).

To show that every Hilbert $A$-module $H$ is $*$-projective, 
let $H_1$ and $H_2$ be Hilbert $A$-modules, $\phi: H\to H_1$ and $s: H_2\to H_1$ 
be bounded module maps with adjoints, where  $s$ is surjective. 
We now apply the statement of (1).

  Note that  (1) holds for any Hilbert $A$-module, in particular, 
it holds for $H_1$ (in place of $H$). 
Since  $s$ is surjective,   we obtain 
$\psi_1\in L(H_1, H_2)$ such that
\beq\label{1proj-8}
s\circ \psi_1={\rm id}_{H_1}.
\eneq
Define $\psi=\psi_1\circ \phi: H\to H_2.$  Then $s\circ \psi=s\circ \psi_1\circ \phi=\phi.$

\end{proof}

\begin{rem}\label{R1}

(i) It is probably worth noting that, in  part (1) of Theorem \ref{1proj}, 
$\|\psi\|=\||ss^*|^{-1/2}\|$ (the inverse is taken in $L(H)$),  and, for the second part, $\|\psi\|=\||ss^*|^{-1/2}\circ \phi\|.$  This, probably,  may  not be improved.

The next three propositions  are easy facts, but perhaps, not entirely trivial.
We include here for the clarity of our further discussion. 
After an earlier version (\cite{LinCuntz}) of this  note was first posted (in 2010), Leonel Robert informed us  
that, using Proposition \ref{1proj} above, he has a proof that the converse of the following also 
holds, i.e., if $H$ is algebraically finitely generated, then $K(H)$ has an identity (see 
Proposition \ref{Probert}). 
\end{rem}

\begin{prop}\label{PH}
Let $A$ be a \CA\, and $H$ be a Hilbert $A$-module. Suppose that
$1_{H}\in K(H).$ Then $H$ is algebraically finitely generated.
\end{prop}

\begin{proof}
Let $F(H)$ be the linear span of  rank one module maps of the form
$\theta_{\xi, \zeta}$ for $\xi, \zeta\in  H.$ Since $F(H)$ is dense in $K(H),$ there is $T\in
F(H)$ such that
\beq\label{PH-1}
\|1_H-T\|<1/4.
\eneq
It follows that $T$ is invertible. 
There are $\xi_1, \xi_2,...,\xi_n,
\zeta_1,\zeta_2,...,\zeta_n\in H$ such that
\beq\label{PH-3}
T(\xi) =\sum_{j=1}^n \xi_j<\zeta_j, \xi>\tforal \xi\in H.
\eneq
But $TH=H.$ This implies that $\sum_{j=1}^n\xi_jA=H.$

\end{proof}


\begin{df}\label{DprojM}
Let $A$ be a \CA\, and $H$ be a Hilbert $A$-module.
We say $H$ is projective (with bounded module maps as morphisms), if for any Hilbert modules $H_1$ and $H_2$
and {{any}}  bounded surjective module {{map}} $s: H_2\to H_1$  and {{any}}
bounded module map 
$\phi: H\to H_1,$ there is a bounded module map
$\psi: H\to H_2$ such that $s\circ \psi=\phi.$ So we have the following commutative 
diagram:
\begin{center}
\begin{large}
\begin{equation}
\begin{tikzcd}
H \ar[d, "\phi"'] \ar[dr, dashed, "\psi"]\\
H_1 & H_2 \ar[l, two heads, "s"]
\end{tikzcd}
\end{equation}
\end{large}
\end{center}                     
 
\end{df}

\begin{prop}\label{FHP}
Let $A$ be a \CA\, and $H$ be a Hilbert $A$-module for which
$K(H)$ has an identity. Then $H$ is a projective Hilbert $A$-module.
\end{prop}

\begin{proof}
Denote by $\tilde A$ the minimum unitization of $A.$
Note that $H$ is a Hilbert $\tilde A$-module. 
We may write $\tilde A^{(N)}=e_1\tilde A\oplus e_2 \tilde A\oplus \cdots \oplus e_N\tilde A,$
where $\la e_i, e_i\ra=1_{\tilde A}.$ 
Let us first show that $\tilde A^{(N)}$ is a projective Hilbert $\tilde A$-module.

Suppose that $H_1$ and $H_2$  are Hilbert $\tilde A$-modules, $s: H_2\to H_1$ is a surjective 
bounded module map and $\phi: \tilde A^{(N)}\to H_1$ is a bounded module map. 
Let $x_i=\phi(e_i),$ $1\le i\le N.$ 
Since $s$ is surjective, 
one chooses $y_i\in H_2$ 
such that
\beq
s(y_i)=x_i,\,\,1\le i\le N.
\eneq
Define   a module map $\psi: {\tilde A}^{(N)}\to H_1$ 
$$
\psi(h)=\sum_{i=1}^N y_i\la e_i, h\ra\rforal h=\oplus_{i=1}^Na_i,
$$
where $a_i\in A,$ $1\le i\le N.$ 
Then $\phi\in K(H, H_2).$ 
Moreover
\beq
s\circ \psi(e_i)=x_i\la e_i, e_i\ra=x_i=\phi(e_i),\,\,1\le i\le N.
\eneq
It follows that $s\circ \psi=\phi.$
%

Now we consider the general case.
 From \ref{PH}, $H$ is algebraically 
finitely generated. Therefore, a theorem of Kasparov (\cite{K}) shows that
$H=PH_{\tilde A}$ for some projection $P\in L(H_{\tilde A}).$ The fact that $1_H\in
K(H)$ implies that $P\in K(H_{\tilde A}).$ Therefore there is an integer
$N\ge 1$ and a projection $P_1\in M_N(\tilde A)$ such that $PH$ is
unitarily equivalent to $P_1H_{\tilde A}.$  In other words, one may assume
that $H$ is a direct summand of ${\tilde A}^{(N)}.$ 

Write ${\tilde A}^{(N)}=H\oplus H'$ for some Hilbert ${\tilde A}$-module $H'.$ 
Let $H_0$ and $H_1$ be Hilbert $A$-modules, $s: H_0\to H_1$ be a surjective 
bounded module map and $\phi: H\to H_1$ be another bounded module map.
We view them as ${\tilde A}$-modules and ${\tilde A}$-module maps. 
Let $\iota: H\to H\oplus H'$ be the embedding and $j: H\oplus H'\to H$ be the projection.
Then $j\circ \iota={\rm id}_{H}.$
We then have the following diagram:
$$
\begin{array}{ccccc}
H&\leftarrow_{j}& H\oplus H'&\leftarrow_{\iota} H\\
\downarrow_{\phi}&&&\\
H_1 &{\large{\twoheadleftarrow_{s}}}& H_0  &
\end{array}
$$
Recall that these are Hilbert $\tilde A$-modules and  bounded $\tilde A$-module maps.
Since we have shown $H\oplus H'={\tilde A}^{(N)}$  is projective, we obtain a bounded 
module map $\psi_1: H\oplus H'\to H_1$ such that the following commutes:
$$
\begin{array}{ccccc}
H&\leftarrow_{j}& H\oplus H'&\leftarrow_{\iota} H\\
\downarrow_{\phi}&&\downarrow_{\psi_1}&\\
H_1 & {\large{\twoheadleftarrow_{s}}} & H_0  
\end{array}
$$
So $s\circ \psi_1=\phi\circ j.$ 
Define $\psi=\psi_1\circ \iota.$  So $s\circ \psi=s\circ \psi_1\circ \iota=\phi\circ j\circ \iota=\phi.$
%
%
\end{proof}

\begin{rem}

(1) One may notice  that we do not provide an estimate of $\|\psi\|$ in Proposition \ref{FHP}.
Note that $s$ induces a one-to one and surjective  Banach module map $\tilde s: H_2/{\rm ker}s\to H_1.$
Therefore its inverse $\tilde s^{-1}$ is bounded. 
In the proof,  for any $\ep>0,$ one may choose $y_i\in H_2$ such 
that $\|y_i\|\le \|\tilde s^{-1}\circ \phi\|+\ep/\sqrt{N}.$ Then one may be able to estimate that
$\|\psi\|\le \sqrt{N}\|\tilde s^{-1}\circ \phi\|+\ep.$ It is not clear one can do better than that in 
these lines of proof.

(2) The fact that $\la e_i, e_i\ra =1_{\tilde A}$ is crucial in the proof. It should be
noted that, when $A$ is not unital,  the argument  in the proof of Proposition \ref{FHP} does not
imply  that $A^{(n)}$ is projective (with bounded module maps as
morphisms). However, we do not assume that $A$ is unital in Proposition \ref{FHP}
(but $K(H)$ is unital). 
On the {{other}} hand, one should be warned that, if $A$ is not unital and $H=A^{(n)},$ 
then $K(H)=M_n(A)$ which is never unital.
There are projective Hilbert modules
for which $K(H)$ is not unital. In fact, 
in  Theorem \ref{BLT1}
 below,  we show 
that $A^{(n)}$ is always projective when $A$ is separable.  

(3) Let $H_1=\overline{xA}$ and $H_2=\overline{yA}.$ Define  a module map $\psi: xA\to yA$
by $\phi(xa)=ya$ for all $a\in A.$  In general such a module map may not be bounded.
Consider $A=C([0,1])$ and $x(t)=t$ and $y=1_A.$ 
Let  $\psi: xA\to yA$ be defined by $\psi(x f)=f$ for all $f\in C([0,1]).$ This module map is not bounded.
We do not know that $A^{(n)}$ is projective in general 
when $A$ is not unital and not separable (see Lemma \ref{appprojL}).


(4)  One may also mention that if 
$H_1, H_2,...,H_n$ are projective Hilbert $A$-modules 
then 
$\bigoplus_{i=1}^n H_i$ is also projective.
\end{rem}

\begin{prop}[L. Robert]\label{Probert}
Let $A$ be a \CA\, and $H$ be an algebraically finitely generated Hilbert $A$-module. 
Then $K(H)$ is unital and $H$ is self-dual.
\end{prop}

\begin{proof}
Write $H=\sum_{i=1}^n\xi_i \tilde A$ (where ${{\xi_i}}\in H$). 
Put $H_1=\tilde A^{(n)}.$   We view both $H$ and $H_1$ as Hilbert $\tilde A$-modules.
Write $H_1=\bigoplus_{i=1}^n e_i \tilde A,$ where $e_i=1_{\tilde A},$ $1\le i\le n.$ 
Define a module map $s: H_1\to H$ by $s(e_i)=\xi_i,$ $1\le i\le n.$ 
Then 
\beq
s(h)=\sum_{i=1}^n \xi_i\la e_i, h\ra\rforal h\in H_1.
\eneq
In other words $s=\sum_{i=1}^n \theta_{\xi_i, e_i}\in K(H_1, H).$
In particular, $s\in L(H_1, H).$ 
Since $H$ is algebraically generated by $\{\xi_1, \xi_2,..,\xi_n\},$ $s$ is surjective.
We then proceed the same proof of Theorem \ref{1proj} to obtain the bounded module map $L_1=T^*S.$
 Note, in this case, $T\in K(H_1\oplus H)$ and hence $L_1\in K(H_1\oplus H).$
 It follows that $s\circ \psi_1\in K(H)$ (with notation in the proof of \ref{1proj}).
  Since $s\circ \psi_1={\rm id}_H,$ we conclude 
 that $K(H)$ is unital. 
 
 By Theorem \ref{1proj}, $\psi_1(H)$ is  an orthogonal summand of 
 $H_1=\tilde A^{(n)}.$  Let $U: H\to  \psi_1(H)$ be a bounded module map 
 which implements the unitary equivalence.  Note that $U\in L(H, \psi_1(H)).$
 Let $Q: H_1\to \psi_1(H)$ be the projection 
 of $H_1$ onto the orthogonal summand $\psi_1(H).$ 
 
 Suppose that $f\in H^\sharp,$ i.e., $f: H\to A$ is  a bounded module map.
 Define $\tilde f: H_1\to A$ by $\tilde f(h)=f\circ U^{-1}\circ Q(h))$ for all $h\in H_1.$
 Note that $f(x)=\tilde f(U(x))$ for all $x\in H.$
 Since $H_1=\tilde A^{(n)}$ is self-dual, there is $g\in H_1$ such 
 that $\tilde f(h)=\la g, h\ra$ for all $h\in H_1.$ If $x\in H,$ then 
 $f(x)=\tilde f(U(x))=\la g, Q(U(x))\ra.$ Hence $f(x)=\la (QU)^*(g), x\ra$ for all $x\in H.$
 Since $(QU)^*(g)\in H,$ this shows that $H=H^\sharp$ and $H$ is self-dual.
 %
 \end{proof}

\begin{lem}[cf. Proposition 1.4.5 of \cite{Pedbook}]\label{Lpolord}
Let $A$ be a \CA\, and 
$H$ a Hilbert $A$-module and $x\in H.$ Then, for any $0<\af<1/2,$ 
there exists $y\in \overline{xA}\subset H$ with 
$\la y, y\ra^{1/2}=\la x, x\ra^{1/2-\af}$ such 
that $x=y\cdot \la x, x\ra^\af.$
\end{lem}
 
 \begin{proof}
 We first consider Hilbert $\tilde A$ module $\overline{x \tilde A}.$ 
 Put $a=\la x, x\ra,$ $a_n=(1/n+a)^{-\af}$ and $x_n=x\cdot (1/n+a)^{-\af},$ $n\in \N.$
 We will show that $\{x_n\}$ is a Cauchy sequence.
 
  Put $\bt_{n,m}=(1/n+a)^{-\af}-(1/m+a)^{-\af},$
 $n,m\in \N.$
 Then
 $$
 \|x \bt_{n,m}\|=\|(\bt_{n,m}\la x, x\ra \bt_{n,m})^{1/2}\|=\|a^{1/2} \bt_{n,m}\|.
 $$
 Since $a^{1/2}(1/n+a)^{-\af}$ converges to $a^{1/2-\af}$ in norm. 
 we conclude  that $\|x\bt_{n,m}\|\to 0$ as $n,m\to\infty.$ It follows that $x_n\to y_\af$ for some 
 $y_\af\in H.$ Hence  $x_n \la x, x\ra^\af\to y_\af \la x, x\ra^\af.$ Moreover,
 $\la y_\af, y_\af \ra^{1/2} =\lim_{n\to\infty} a^{1/2} a_n=a^{1/2-\af}$ (converge in norm).
 
 Similarly,
 $\lim_{n\to \infty}\|x-xa_na^\af\|=\lim_{n\to\infty}\|a^{1/2}(1-a_na^\af)\|=0.$
 Since $x_na^\af=xa_n a^\af,$ we obtain that $\lim_{n\to \infty}x_na^\af=x.$ It follows that $x=y_\af\la x, x\ra^{\af}.$
 Choose $\af<\af'<1/2.$  Then $(1/n+a)^{-\af'}a^{\af'-\af}\in A.$
 Since $\lim_{n\to\infty} x(1/n+a)^{-\af'}a^{\af'-\af}=\lim_{n\to\infty} x(1/n+a)^{-\af'}(1/n+a)^{\af'-\af}=\lim_{n\to\infty}x(1/n+a)^{-\af}=y_\af,$ $y_\af\in \overline{xA}.$

\end{proof}

Let us end this section with the following clarification on Definition \ref{Dgen}.

\begin{cor}\label{Rgen}
Let $H$ be a Hilbert $A$-module and ${\cal F}\subset H$ be a subset.

(1) Then 
${\cal F}\subset \overline{\{z\cdot a: z\in {\cal F},\, a\in A\}}.$

(2) If $H$ is algebraically finitely generated, then there are $x_1, x_2,..., x_n\in H$
such that $H=\{\sum_{i=1}^n x_i a_i: a_i\in A\}$ (see Definition \ref{Dgen}).
\end{cor}

\begin{proof}
For (1), let $z\in {\cal F}.$  By Lemma \ref{Lpolord}, we may write
$z=y\la z, z\ra^\af$ for some $0<\af<1/2$ and $y\in H.$
Let $\{e_\lambda\}\subset A$ be an approximate identity.
Then  $z e_\lambda=y\la z, z \ra^\af e_\lambda\to   y\la z, z\ra=z.$
Hence $z\in \overline{\{z\cdot a: z\in {\cal F},\, a\in A\}}$
(So in 
Definition \ref{Dgen}, ${\cal F}$ is in the closure of $\{za: z\in {\cal F}\andeqn a\in A\}.$)

For (2), suppose that $H=\sum_{i=1}^n x_i{\tilde A}.$ 
By Lemma \ref{Lpolord}, we may write
$x_i=y_i \la x_i, x_i\ra^\af$ for some $0<\af<1/2$ and $y_i\in H,$
$1\le i\le n.$
Let $H_0=\sum_{i=1}^n y_i A\subset H.$
Since $x_i=y_i\la x_i, x_i\ra^\af\in H_0,$ $H\subset H_0.$
It follows that $H=\sum_{i=1}^n y_i A.$
\end{proof}

%


\section{Countably generated Hilbert modules}

Let $A$ be a \CA\, and $H_1$ and $H_2$ be a pair of Hilbert $A$-modules. 
In  this section  we  discuss the question when  $H_1$ is unitarily equivalent to a Hilbert $A$-submodule of $H_2.$ 
The main result of this section is Theorem \ref{PQ}  which 
provide some answer to the question. For the connection to Cuntz semigroup, see Corollary \ref{pq}.

\begin{lem}\label{Lappunit}
Let $A$ be a \CA\, and $H$ a Hilbert $A$-module.
Let $\{E_\lambda\}$ be an approximate identity for $K(H).$ Then, for any 
$\ep>0$ and any finite subset ${\cal F}\subset H,$ there is $\lambda_0$
 such 
that, for all $\lambda\ge \lambda_0,$ 
\beq
\|E_\lambda(x)-x\|<\ep.
\eneq

\end{lem}

\begin{proof}
Let ${\cal F}=\{\xi_1, \xi_2,...,\xi_k\}$ and $\ep>0.$
Without loss of generality, we may assume that $0\not=\|\xi_i\|\le 1,$ $1\le i\le k.$
There is $1/3<\af<1/2$ such that
\beq
\|\la \xi_i, \xi_i\ra^{1-\af}-\la \xi_i,\xi_i\ra^\af\|=\|\la \xi_i, \xi_i\ra^\af(\la \xi_i,\xi_i\ra^{1-2\af}-1)\|<\ep/3,\,\,\,1\le i\le k.
\eneq
By Lemma \ref{Lpolord}, write $\xi_i=\zeta_i\la \xi_i, \xi_i\ra^\af$ with $\la \zeta_i, \zeta_i\ra=\la \xi_i, \xi_i\ra^{1-2\af},$ $1\le i\le k.$
Put $S_i=\theta_{\zeta_i, \zeta_i},$ $1\le i\le k.$ 
Then, for $1\le i\le k,$ 
\beq\nonumber
\|S_i(\xi_i)-\xi_i\|=\|\zeta_i\la \zeta_i, \xi_i\ra-\xi_i\|=\|\zeta_i\la \xi_i,\xi_i\ra^{1-\af}-\zeta_i\la \xi_i, \xi_i\ra^\af\|<\ep/3.
\eneq
There exists $\lambda_0$ such that, for all $\lambda\ge \lambda_0,$
\beq
\|E_\lambda S_i-S_i\|<\ep/3,\,\,1\le i\le k.
\eneq
It follows that, for all $\lambda\ge \lambda_0,$
\beq\nonumber
\|E_\lambda(\xi_i)-\xi_i\|&\le &\|E_\lambda(\xi_i)-E_\lambda S_i(\xi_i)\|+\|E_\lambda S_i(\xi_i)-S_i(\xi_i)\|+
\|S_i(\xi_i)-\xi_i\|\\\nonumber
&<&\|E_\lambda\|\|\xi_i-S_i(\xi_i)\|+\ep/3+\ep/3<\ep,\,\,\hspace{0.2in}1\le i\le k.
\eneq
\end{proof}

At least part of the following is known, in fact,  (1) $\Leftrightarrow$ (2) is known as Corollary 1.5 of \cite{MP}. 
We also think that the rest of them are known.
Note that even $H$ is countably generated 
it may not be separable when $A$ is not separable.

\begin{prop}[cf. Corollary 1.5 of \cite{MP}]\label{TKHsigma}
Let $A$ be a \CA\, and $H$ be a Hilbert $A$-module.
Then the following are equivalent.

(1) $H$ is countably generated,

(2) $K(H)$ is $\sigma$-unital and 

(3) There exists an increasing sequence of module maps $E_n\in K(H)_+$ with $\|E_n\|\le 1$ such that
\beq
\lim_{n\to\infty} \|E_nx-x\|=0\tforal x\in H.
\eneq

\end{prop}

\begin{proof}
As mentioned above, by Corollary 1.5 of \cite{MP}, (1) $\Leftrightarrow$ (2). 

(1) $\Rightarrow$ (3): Let $\{x_n\}$ be a sequence of elements in the unit ball of $H$ 
such that $H$ is generated by $\{x_n\}$ as Hilbert $A$-module.
Let ${\cal F}_n$ be the set of those elements with the form $x_ia,$ $1\le i\le n$ and $a\in A$ with 
$\|a\|\le n.$  Note that $\cup_{n=1}^\infty {\rm span}({\cal F}_n)$ is dense in 
$H.$ Warning: ${\cal F}_n$ is not a finite subset.

For each $n\in \N,$ there exists $0<1/3<\af_n<1/2$ such that, for any $0\le b\le 1$ in $A,$ 
\beq\label{TKsigma-5}
\|b^{1/2-\af_n}(1-b^{1-2\af_n})b^{\af_n}\|\le 1/n^2.
\eneq
By Lemma \ref{Lpolord}, choose $\xi_{i,n}\in \overline{x_i A}$ such that
$x_i=\xi_{i,n}\la x_i, x_i\ra^{\af_n}$ with $\la \xi_{i,n}, \xi_{i,n}\ra=\la x_i, x_i\ra^{1-2\af_n}.$ Put $b_i=\la x_i,x_i\ra,$
$n\in \N.$ 
Then
$$
\theta_{\xi_{i,n},\xi_{i,n}}(x_ia)=\xi_{i,n}\la \xi_{i,n}, \xi_{i,n}b_i^{\af_n}\ra a=
\xi_{i,n}\la \xi_{i,n}, \xi_{i,n}\ra b_i^{\af_n} a=\xi_{i,n}b_i^{1-\af_n} a.
$$
Then, for any $\|a\|\le n$ and $i,n\in \N,$  by \eqref{TKsigma-5},
\beq
\|\theta_{\xi_{i,n},\xi_{i,n}}(x_ia)-x_ia\|^2&=&\|\xi_{i,n}(b_i^{1-\af_n}-b_i^{\af_n})a\|^2=\|\xi_{i,n}(b_i^{1-2\af_n}-1)b_i^{\af_n}a\|^2\\
&=&
\|a^*b_i^{\af_n}(b_i^{1-2\af_n}-1)\la \xi_{i,n}, \xi_{i,n}\ra (b_i^{1-2\af_n}-1)b_i^{\af_n} a\|\\\label{TKsigma-6}
&=&\|b_i^{1/2-\af_n}(1-b_i^{1-2\af_n})b_i^{\af_n}a\|^2\le (1/n)^4 \|a\|^2<1/n^2.
\eneq
By using an approximate identity of $K(H),$ we obtain an increasing sequence 
of $\{E_n\}\in A_+$ with $\|E_n\|\le 1$ such that
\beq\label{TKsigma-10}
\|E_n \theta_{\xi_{i,n} \xi_{i,n}}-\theta_{\xi_{i,n}, \xi_{i,n}}\|<1/n^2,\,\, 1\le i\le n.
\eneq
It follows that, for any $m\ge n$ and $a\in A$ with $\|a\|\le n,$  by \eqref{TKsigma-6} and \eqref{TKsigma-10},
\beq\nonumber
\hspace{-0.3in}\|E_m(x_ia)-x_ia\|&\le &\|E_m(x_ia-\theta_{\xi_{i,m},\xi_{i,m}}(x_ia))\|+
\|E_m\circ \theta_{\xi_{i,m}, \xi_{i,m}}(x_ia)-\theta_{\xi_{i,m}, \xi_{i,m}}(x_ia)\|\\\nonumber
&&\hspace{0.4in}+
\|\theta_{\xi_{i,m} \xi_{i,m}}(x_ia)-x_ia\|\\
&<& 1/m+\|xa\|/m^2+1/m^2<2/m+1/m^2.
\eneq
Since $\cup_{i=1}^\infty {\rm span} ({\cal F}_i)$ is dense in $H,$ this proves (3). 

(3) $\Rightarrow$ (2): 
Let $\{E_n\}$ be an increasing  sequence in $K(H)_+$ with $\|E_n\|\le 1$ such that
\beq
\lim_{n\to\infty}\|E_n(x)-x\|=0\rforal x\in H.
\eneq
Fix $x_1, x_2,...,x_N, y_1,y_2,...,y_N\in H.$ Let $S=\sum_{j=1}^N \theta_{x_j,y_j}$ and 
$$M=N\cdot \max\|\{\|y_j\|; 1\le j\le N\}.$$ 
For any $\ep>0,$ there exists   $n_0\in \N$    such that
\beq
\|E_m(x_j)-x_j\|<{\ep\over{(M+1)}},\,\,1\le j\le N.
\eneq
whenever $m\ge n_0.$ 
Then, for any $m\ge n_0$ and, for any $z\in H,$  
\beq
\|E_m S(z)-S(z)\|=\|\sum_{j=1}^N (E_m x_j-x_j)\la y_j, z\ra\|\le ({\ep\over{ (M+1)}})\sum_{j=1}^m\|y_j\|\|z\|
<\|z\|\ep.
\eneq
It follows that $\lim_{n\to\infty}\|E_mS-S\|=0.$  Since $F(H)$ is dense in $K(H),$ this implies that $K(H)$ is $\sigma$-unital.
\end{proof}

\begin{df}\label{Dwdense}
Let $H$ be a Hilbert $A$-module and $G\subset H$ be a subset.
We say $G$ is $A$-weakly dense in $H,$ if, for any  finite subset 
${\cal F}
\subset H,$ $\ep>0,$ 
and any $x\in H,$ there is $g\in G$ 
such that
\beq
{\|} \la  (x-g), z\ra \|<\ep\tforal z\in {\cal F}.
\eneq
(see Remark \ref{remaweakdense}).
\end{df}

One may compare the next proposition with Lemma 1.3 of \cite{MP}.

\begin{prop}\label{cdense}
Let $A$ be a $\sigma$-unital \CA\, and $H$ be a Hilbert $A$-module.

(1)  Let $s\in K(H)_+.$
Then  $s$ is a strictly positive element of $K(H)$ if  
$s(H)$ is $A$-weakly dense in $H,$   and, if $s$ is a strictly positive element of $K(H),$
then $s(H)$ is dense in $H.$ 

(2) Let $H_1$ and $H_2$ be  Hilbert $A$-modules and $T\in K(H_1, H_2)$
such that $T(H_1)$ is $A$-weakly dense in $H_2,$ then $TT^*\in K(H_2)$ is a strictly positive element
of $K(H_2).$

\end{prop}

\begin{proof}
For (1),  we note that, by Lemma 1.3 of \cite{MP}, $s$ is a strictly positive element if and only $s(H)$ is dense in $H.$ 
%

We now suppose that    $s(H)$ is $A$-weakly dense in $H.$ 
{{\Wlog, we may assume that $\|s\|  \le 1.$}}
Let $p$ be the open projection associated with $s,$ i.e., $s^{1/n}\nearrow p.$
Put $\eta_n=\|s^{1/n}s-s\|,$ $n\in \N.$ 
Then $\lim_{n\to\infty}\eta_n=0.$

Fix  $\ep>0. $  Let $x, y\in H$ such that $\|x\|,\, \|y\|\le 1.$
Choose $\xi_1, \xi_2\in H$ such that $\|\xi_1\|, \, \|\xi_2\|\le 1.$
Then, there is $\xi\in H$ such that
\beq
\| \la s(\xi)-\xi_1,  y\ra \| <\ep/2.
\eneq
Then
\beq
\| \la \theta_{\xi_1, \xi_2}(x)-\theta_{s(\xi), \xi_2}(x), y\ra\|=\|{{\la}} \xi_1\la \xi_2, x\ra, y\ra-\la s(\xi)\la \xi_2, x\ra , y\ra \|\\
=\| \la  \xi_2, x\ra^* \la \xi_1-s(\xi), y\ra \|<\ep/2.
\eneq
We estimate that 
\beq\nonumber
\| \la s^{1/n} \theta_{\xi_1,\xi_2}(x)-\theta_{\xi_1, \xi_2}(x), y\ra\|&=&\| \la s^{1/n} (\theta_{\xi_1, \xi_2}-\theta_{s(\xi), \xi_2})(x), y\ra\| \\\nonumber
&&\hspace{-0.6in}+\|\la (\theta_{s^{1/n}s(\xi), \xi_2}(x)-\theta_{s(\xi_1), \xi_2}(x)), y\ra \|+ \| \la (\theta_{s(\xi), \xi_2}(x)-\theta_{\xi_1, \xi_2}(x)), y\ra \|\\\nonumber
&<&\|s^{1/n}\|\|(\theta_{\xi_1, \xi_2}-\theta_{s(\xi), \xi_2})(x), y\ra\| +\eta_n+\ep/2\\
&<&\ep+\eta_n.
\eneq
Fix $n\in \N,$ 
let $\ep\to 0.$ We obtain that
$$
\| \la s^{1/n}\theta_{\xi_1,\xi_2}(x)-\theta_{\xi_1, \xi_2}(x), y\ra\|\le \eta_n.
$$
This holds for all $y\in H$ with $\|y\|\le 1.$
It follows that
\beq\label{36-1}
\|s^{1/n} \theta_{\xi_1,\xi_2}(x)-\theta_{\xi_1, \xi_2}(x)\|<\eta_n,\,\,n\in \N.
\eneq
Since \eqref{36-1}  holds for every  $x\in H$ with $\|x\|\le 1,$ we obtain that 
\beq
\|s^{1/n} \theta_{\xi_1,\xi_2}-\theta_{\xi_1, \xi_2}\|<\eta_n,\,\,n\in \N.
\eneq
Since linear combinations of module maps of the form $\theta_{\xi_1, \xi_2}$ is dense in $K(H),$ 
we obtain, for any $k\in K(H),$
\beq
\lim_{n\to\infty}\|s^{1/n}k-k\|=0=\lim_{n\to\infty}\|k^{1/2} (1-s^{1/n})k^{1/2}\|. 
\eneq
It follows that $\|k^{1/2}(1-p)k^{1/2}\|=0,$ or $pk=k$ for all $k\in K(H).$
Thus
$p$ is the open projection associated with the hereditary \SCA\, $K(H)$ (see Lemma 2.13 of \cite{Lninj}). It then follows that $s$
is a strictly positive element of $K(H).$

For (2), let $H=H_1\oplus H_2.$   Define $\tilde T\in K(H)$ by 
$\tilde T(h_1\oplus h_2)=0\oplus T(h_1)$ for $h_1\in H_1$ and $h_2\in H_2.$
Let $\tilde T=V(\tilde T^*\tilde T)^{1/2}$ be the polar decomposition 
of $\tilde T$ in $K(H)^{**}.$  Then we may write $T=\tilde T|_{H_1}=V(T^*T)^{1/2}.$
Since $T\in K(H_1, H_2),$ $T^*\in K(H_2, H_1).$

Recall that $T^*=|T|V^*,$ $V|T|V^*=|TT^*|^{1/2}$ and $V^*|TT^*|^{1/2}=T^*.$
Hence $T=(TT^*)^{1/2} V.$ 
We then write  $T=(TT^*)^{1/4}(TT^*)^{1/4}V.$  Note that $(TT^*)^{1/4}V\in K(H)$
and $(TT^*)^{1/4}V(H)\subset H_2.$  It follows that 
$(TT^*)^{{1/4}}(H_2)$ is  $A$ -weakly dense in $H_2,$ as $T(H)=(TT^*)^{1/4}((TT^*)^{1/4}V(H))$ is 
$A$-weakly dense in 
$H_2.$ 
 By (1), this implies that $(TT^*)^{1/4}$ is a strictly positive element 
of $K(H_2).$ Hence $TT^*$ is a  strictly positive element 
of $K(H_2).$
\end{proof}

\begin{NN}\label{RN}
{\rm Let $A$ be a \CA, $H_1, H_2$ be Hilbert $A$-modules and $T\in B(H_1, H_2).$ 
Note $T^*\in B(H_2, H_1^\sharp)$ and, for any $y\in H_1,$ $T^*(y)(x)=\la y, Tx\ra $ for all
$x\in H_1.$  We may view $H_1\subset H_1^\sharp.$ We say $T^*(H_2)$ is dense 
in $H_1,$ if $\overline{T^*(H_2)}\supset H_1$ (even though $T^*(H_2)$ may not be in $H_1$).
Recall that $TK_0^*\in K(H_1, H_2),$ for any $K_0\in K(H_1)$ which has an adjoint
$(TK_0^*)^*=K_0T^*\in K(H_2, H_1).$ 
If we write $T\in LM(K(H_1, H_2)),$ then,
in general, $T^*\in RM(K(H_2, H_1)).$   In particular, $K_0T^*\in K(H_2, H_1)$
for all $K_0\in K(H_1).$ These will be used in the next theorem.}

\end{NN}

The equivalence of (a) and (d) in part (2) of next theorem  is known (see Theorem 4.1 of \cite{mF}). 

\begin{thm}\label{PQ}
Let $A$ be a \CA\, and $H_1$ and $H_2$ be Hilbert $A$-modules.
Suppose that $H_2$ is countably generated.

(1) Then $H_2$ is unitarily equivalent to a Hilbert $A$-submodule of $H_1$ if and only 
if there is  bounded module map
$T: H_1\to H_2$ such that $T(H_1)$ is $A$-weakly dense in $H_2.$ 

(2)   Suppose that $H_1$ is also countably generated.
Then the following are equivalent:

(a) $H_1$ and $H_2$ are unitarily equivalent as Hilbert $A$-module;

(b) there 
exists a bounded module map $T\in B(H_1, H_2)$ such that $T(H_1)$ is
$A$-weakly  dense in $H_2$ and $(K_0T^*)(H_2)$ is $A$-weakly dense in 
$H_1$ for any strictly positive element $K_0\in K(H_1);$

(c) there 
exists a bounded module map $T\in B(H_1, H_2)$ such that $T(H_1)$ is dense in $H_2$ and $T^*(H_2)$ is dense in 
$H_1$ (see  \ref{RN});

(d) There is an invertible bounded module map $T\in B(H_1, H_2).$ 
\end{thm}

\begin{proof}

For (1), the ``only if" part is clear. 
Let us  assumes that there is $T\in B(H_1, H_2)$ whose range is $A$-weakly dense in $H_2.$

Suppose that $H_2$ is generated by $\{x_1, x_2,...,x_n,...\}.$ 
For each $n\in \N,$ there exists a sequence $\{y_{n,k}\}$ in $H_2$ such that,
$$
\|\la T(y_{n,k})-x_n, x_i\ra\| <({1\over{k+n}})\max\{\|x_i\|: 1\le i\le n+k\},\,\, 1\le i\le n+k.
$$
Hence, for any fixed $n$ and,  any $y\in H,$ $\lim_{k\to \infty}\|\la T(y_{n,k})-x_n, y\ra \|=0.$
It follows that the set of linear combinations of  elements in $\{T(y_{n,k})a: n, k\in \N, a\in A\}$ is $A$-weakly dense in $H_2.$ 
Let $H_3\subset H_1$ be the closure of $A$-submodule  generated by 
$\{y_{n,k}: k, n\in \N\}.$ Then $H_3$ is a countably generated Hilbert $A$-submodule 
of $H_1$ such that $T(H_3)$ is $A$-weakly dense in $H_2.$

Put $H=H_3\oplus H_2.$ In what follows we may  identify $T$ with  the map $\tilde T(h_3\oplus h_2)=0\oplus T(h_3)$ for $h_3\in H_3$ and $h_2\in H_2,$   whenever it is convenient.
We may also view $K(H_2)$ and $K(H_3)$ as hereditary \SCA s of $K(H)$ (see Lemma 2.13 of \cite{Lninj}). 

Applying Corollary 1.5 of \cite{MP} (see 
Lemma \ref{TKHsigma} above for convenience), we obtain   strictly positive elements $K_0\in K(H_3)\subset K(H)$ and 
$L\in K(H_2)\subset K(H),$  respectively.
 If follows from (1) of Proposition  \ref{cdense} that $K_0(H_3)$ is dense in 
$H_3.$ 
%
%
Thus   $TK_0(H_3)$ is $A$-weakly dense in $H_2.$ 
Applying part (2) of Proposition \ref{cdense}, 
we conclude that 
 $S=(TK_0)(TK_0)^*$ is a strictly positive element of $K(H_2).$
 %
 %
%
%

%
Therefore the range projection $P_s$  of $S$ in $K(H_3)^{**}$ is the same as the range projection  
of $L.$ 
Write $(TK_0)^*=US^{1/2}$ in the polar decomposition of $(TK_0)^*$ in $K(H)^{**}.$ 
Note  that $(TK_0)^*=K_0^{1/2}(TK_0^{1/2})^*=US^{1/2}.$   
It follows that $U(S^{1/2}(H))\subset H_3.$ Since $S^{1/2}$ is a strictly positive element
of $K(H_2),$ $H_2=\overline{S^{1/2}H_2}.$  Hence $U(H_2)\subset H_3.$
Note that $UY\in {\overline{K_0\cdot K(H)}}$ for any 
$Y\in K(H).$   In particular, $US'U^*\in \overline{K_0 \cdot K(H) \cdot K_0}=K(H_3)$ for all $S'\in K(H_2).$
We also have $UH_2\subset H_3.$ 
Moreover
\beq
US^{1/n}U^*\nearrow UP_sU^*=Q,
\eneq
where $P_s$ is the range projection of $S$ (in 
$K(H)^{**}$) and $Q$ is an open projection of $K(H)^{**},$ and  $QUP_s=U.$ 
It follows that $U^*U=P_s$ and $UH_2\subset H_3$ is unitarily equivalent to $H_2.$ 
This proves (1)

For (2). Let us first show that (b) $\Rightarrow$  (a). 
We keep the notation used in the proof of part (1).
But since $H_1$ is now assumed to be countably generated, we choose $H_3=H_1.$ 
In particular, $K_0$ is a strictly positive element of $K(H_1).$
With notation in (1), it suffices to show that $U(H_2)$ is dense in $H_1.$
Since $K_0T^*(H)$ is $A$-weakly dense in $H_1,$ by Proposition \ref{cdense},
$K_0T^*TK_0$ is a strictly positive element in $K(H_1).$  By (1) of Proposition \ref{cdense},
$K_0T^*TK_0(H_1)$ is dense in $H_1.$  Because $TK_0(H_1)\subset H_2,$ this implies 
that $K_0T^*(H_2)$ is actually dense in $H_2.$ 

Since $K_0T^*=US^{1/2},$ and $K_0T^*(H_2)$ is dense in $H_1,$ $U(S^{1/2}(H_2))$ is dense in
$H_1.$ As in the proof of (1) above, $S^{1/2}(H_2)\subset H_2.$ Hence, indeed,
$U(H_2)$ is dense in $H_1.$

Now consider (c) $\Rightarrow$ (b). 
It suffices to show that $K_0T^*(H_2)$ is dense in $H_1.$

Fix $z\in H_1.$  Since $K_0$ is a strictly positive element, by  Proposition \ref{cdense}, $K_0(H_1)$ is dense in 
$H_1.$ There exists a sequence $z_n\in H_1$ such that
\beq
\lim_{n\to\infty} \|K_0(z_n)-z\|=0
\eneq
By the assumption of (c), for each $n\in \N,$ there exists a sequence of elements $y_{n,k}\in H_2$ such that
{{(we do not assume that $T^*(y_{n,k})\in H_1$)}}
\beq
\lim_{k\to\infty}\|T^*(y_{n,k})-z_n\|=0.
\eneq
Thus, one obtain a subsequence $\{y_{m(n)}\}\subset \{y_{n,k}: k,n\in \N\}$ such that
\beq
\lim_{n\to\infty} \|K_0T^*(y_{m(n)})-z\|=0.
\eneq
This shows that $K_0T^*(H_2)$ is dense in $H_1,$ as desired. 

It is clear that (a) $\Rightarrow$ (b), (a) $\Rightarrow$ (c) and (a) $\Rightarrow$ (d).

It remains to  show that (d) $\Rightarrow$ (b). 
We will apply the proof of (1) first. 
But since $H_1$ is countably generated, choose $H_3=H_1.$ 
Let $K_0$ be a strictly positive element of $K(H_1).$
Since $T$ is invertible, $T(H_1)$ is dense in $H_2.$ 
Therefore, it suffices to show that $K_0T^*(H_2)$ is $A$-weakly dense in $H_1.$
Let $H=H_1\oplus H_2.$ We will view $K(H_1)$ and $K(H_2)$ as hereditary \SCA s 
of $K(H).$   So $T$ is identified with the bounded module map 
$(h_1\oplus h_2)\mapsto  0\oplus T(h_1)$ and 
$T^{-1}$ is identified with the bounded module map 
$(h_1\oplus h_2)\mapsto T^{-1}(h_2)\oplus 0.$ 
In what follows, 
 we use the fact that $(T^{-1})^*=(T^*)^{-1}$ and 
we will also work in $H^\sim,$ if necessary (see Definition \ref{DHsim}).
In this way, we may write $T^*, (T^{-1})^*\in RM(K(H)).$ 
Hence, if $E\in K(H_2)\subset K(H),$ $E(T^*)^{-1}=E(T^{-1})^*\in K(H).$ 

Let $\{E_n\}$ be an approximate identity for $K(H_2).$ 
For any finite subset ${\cal F}\subset H_2$  in the unit ball of $H_2,$
and $\ep>0,$ by Lemma \ref{Lappunit}, there is  $n_0\in \N$ such
that, for all $n\ge n_0,$ 
\beq
\|T^{-1}E_n(x)-T^{-1}(x)\|<\ep/(\|T^*\|+1)(\|K_0\|+1)\rforal x\in {\cal F}.
\eneq
It follows that, for any $y\in H_1$ with $\|y\|\le 1,$
\beq
\|\la (E_n(T^*)^{-1}-(T^*)^{-1})(y), x\ra\|&=&\|\la y, (T^{-1}E_n-T^{-1})(x)\ra\|\\
&\le& \|y\|\|\|T^{-1}E_n(x)-T^{-1}(x)\|\\
&<&\ep/(\|T^*\|+1)(\|K_0\|+1)
\eneq
for all $x\in {\cal F}.$
Thus
\beq
\|\la T^*E_n(T^*)^{-1}(y)-y, x\ra \|&=&\|\la T^*E_n(T^*)^{-1}(y)-T^*(T^*)^{-1}(y), x\ra\|\\
&<&\ep/(\|K_0\|+1).
\eneq
Hence
\beq\label{37-10}
\|  \la K_0T^*E_n(T^*)^{-1}(y)-K_0(y), x\ra\|<\ep.
\eneq
Since $E_n(T^*)^{-1}=E_n(T^{-1})^*\in K(H),$  we have $E_n(T^*)^{-1}(H_1)\subset H_2.$ 
{{Also, by  (1) of Lemma \ref{cdense}, $K_0(H_1)$ is dense in $H_1.$}} 
It follows from \eqref{37-10} that $K_0T^*(H_2)$ is $A$-weakly dense in $H_1.$
%
%
%
\end{proof}

\begin{rem}
Theorem \ref{PQ} may provide a correction to 2.2 of \cite{Lninj} (which was not used there).

As the reader may expect that Theorem \ref{PQ} will lead some discussion about the Cuntz semigroups.
Indeed we  will discuss this in next section. But let us end this section with the following corollary.
\end{rem}

\begin{df}
Let $A$ be a \CA\, and  $b\in A_+.$ Denote by ${\rm Her}(b)=\overline{bAb}$ the hereditary \SCA\, of $A.$
Suppose also 
$a\in A_+.$
Let us write $a\simle b$ if there is $x\in A$ such that
$x^*x=a$ and $xx^*\in {\rm Her}(b)$ {{(see \cite{Cu1} and this notation in \cite{Lincsc}).}}
\end{df}

\begin{cor}\label{pq}
Let $A$ be a \CA\, and let  $a, b\in A_+.$   Suppose that
$H_1=\overline{aA}$ and $H_2=\overline{bA}.$ Then 
 $a\simle b,$ if there is $T\in
B(H_2, H_1)$  whose range is $A$-weakly dense, and if $a\simle b,$ then 
there is $T\in B(H_2, H_1)$ 
which has dense range.
\end{cor}

\begin{proof}
Suppose that $a\simle b,$ i.e., there is 
$x\in A$ such that
\beq\label{pq-1-1}
x^*x=a\andeqn xx^*\in  {\rm Her}(b)=\overline{bAb}.
\eneq
Then, for any $c\in A,$ $(x^*bc)(x^*bc)^*=x^*bcc^*bx\in \overline{aAa}.$ It follows that $x^*bc\in \overline{aA}=H_2$
for all $c\in A.$ 
Define $T: H_2\to H_1$ by $T(bc)=x^*bc$ for all $c\in A.$  Then $T\in B(H_2, H_1).$
Since $x^*x=a,$ by Dini's theorem, $x^*b^{1/n}x$ converges to $a$ in norm. Since $T$ is a bounded module map,
we conclude that $T(H_2)$ is dense in $H_1.$

Now we assume that there is $T\in B(H_2, H_1)$ whose range  is $A$-weakly dense in $H_1.$
Note that $a$ is a strictly positive element of $K(H_1)$  and 
$b$ is a  strictly positive element of $K(H_2),$ respectively. 
By  Proposition \ref{cdense}, $b(H_2)$ is dense in $H_2.$ It follows that $Tb(H_2)$ is $A$-weakly dense in $H_1.$
By (2) of Proposition \ref{cdense}, 
$c=(Tb)(Tb)^*$ is a strictly positive element of   ${\rm Her}(a).$
Write $(Tb)^*=uc^{1/2}$ as polar decomposition in $A^{**}.$ 
Then $uy\in A$ for all $y\in \overline{aA}.$ Put $x=ua^{1/2}.$  
Then 
\beq
x^*x=a \andeqn xx^*\in K(H_2)={\rm Her}(b).
\eneq

\end{proof}

\section{Equivalence classes of Hilbert modules and stable rank one}

Now  let us discuss the possibility  to use quasitraces 
to measure Hilbert $A$-modules. 
We will consider the question stated in \ref{q5}. Moreover, we will 
discuss the case that  unitary equivalence classes of Hilbert $A$-submodules can be determined 
by the Cuntz semigroup. 

\begin{df}\label{DCH}
Let $A$ be a $\sigma$-unital \CA. Then one may identify $A\otimes {\cal K}$ with 
$K(H_A).$ 
Denote by ${\rm CH}(A)$ the unitary equivalence classes $[H]$ of 
countably generated Hilbert $A$-modules (where $H$ is a countably generated 
Hilbert $A$-module). 

One may  define $[H_1]+[H_2]$ to be the unitarily equivalent class $[H_1\oplus H_2].$ 
Then ${\rm CH}(A)$ becomes a semigroup.

Let $a,b \in (A\otimes {\cal K})_+.$ We write 
$a{\stackrel{\approx}{\sim}} b,$ if there exists $x\in A\otimes {\cal K}$ such that
$x^*x=a$ and $xx^*$ is a strictly positive element of ${\rm Her}(b).$ 
In the terminology of \cite{ORT}, $a\cuapprox b$ is the same 
as $a\sim b'\cong b$ in \cite{ORT}.

Next proposition states that ``$\cuapprox$"  is an equivalence relation.

Let $a\in (A\otimes {\cal K})_+.$ Define $H_a=\overline{a(H_A)}.$ 
Let $H$ be a countably generated Hilbert $A$-module. Then 
by \cite{MP} (see Proposition \ref{TKHsigma}, for convenience), $K(H)$ is $\sigma$-unital. 
We may view $H$ as a Hilbert $A$-submodule of $H_A=l^2(A).$
Let $b\in K(H)$ be a strictly positive element 
of $K(H).$  Then, by Proposition \ref{cdense}, $H_b:= \overline{bH_A}=\overline{bH}=H.$ If $c\in K(H)$ is
{{another}} strictly positive element, 
then $H_c=H_b.$ 


Let $p_a, p_b$ be the open projections corresponding to $a$  and $b $ (in $(A\otimes {\cal K})^{**}$),
respectively.  Define $p_a\approx^{cu} p_b$ if there is $v\in (A\otimes {\cal K})^{**}$ such that
$v^*v=p_a$ and $vv^*=p_b$ and, for any $c\in {\rm Her}(a)$ and $d\in {\rm Her}(b),$ $vc, v^*d\in A\otimes {\cal K}.$ See also \cite{Cu1} and \cite{Cu2}.

\end{df}

We would like to remind the reader of the following statement.

\begin{prop}[Proposition 4.3 and 4.2 of \cite{ORT}]\label{Prelation}
Let $A$ be a $\sigma$-unital \CA\, and $a, b\in (A\otimes {\cal K})_+.$
Then the following are equivalent.

(1) $a \cuapprox b;$

(2) $p_a\approx^{cu} p_b$ and 

(3)  $[H_a]=[H_b]$ in ${\rm CH}(A).$ 

Moreover, ``\,$\cuapprox$\," is an equivalence relation. 

\end{prop}

\begin{NN}\label{NN0}

{\rm By subsection 4 of \cite{BC}, there are examples of stably finite and separable \CA s 
$A$ which contain 
positive elements $a\sim b$ in ${\rm Cu}(A)$ but 
$a\not\cuapprox b.$}
\end{NN}

\begin{NN}\label{NN1}
{\rm (1) Let $A$ be a $\sigma$-unital \CA\, of stable rank one. By Theorem 3 of \cite{CEI}, 
$a\sim b$ in ${\rm Cu}(A)$ if and only if $H_a$ and $H_b$ are unitarily equivalent. 
One may ask whether  the converse holds.  Theorem \ref{Tstr=1} provides a partial answer to the question. 
}

{\rm 
(2)  
Let $A$ be a $\sigma$-unital simple \CA\, and ${\rm Ped}(A)$ be its Pedersen ideal. 
Let $a\in {\rm Ped}(A)\cap A_+.$ Then ${\rm Her}(a)\otimes {\cal K}\cong A\otimes {\cal K},$ by 
Brown's stable isomorphism theorem (\cite{Brst}).   One notes that ${\rm Her}(a)$
is algebraically simple.}
{\rm To study the Cuntz semigroup 
of $A,$ or the semigroup ${\rm CH}(A),$ one may consider ${\rm Her}(a)$ instead of $A.$ 

(3) Let us now assume that  $A$ is algebraically simple. 
Let $\widetilde{QT}(A)$ be the set of all 2-quasi-traces defined on
 the Pedersen ideal of $A\otimes {\cal K}.$
Let $QT(A)$ be the set of 2-quasi-traces $\tau\in \widetilde{QT}(A)$ such that $\|\tau|_A\|=1.$
 Let us assume that $QT(A)\not=\emptyset.$  
It follows from Proposition 2.9 of \cite{FLosc} that $0\not\in \Qw.$  Recall that $\Qw$ is a compact set (see Proposition 2.9
of \cite{FLosc}).}

{\rm  Denote by $\Aff_+(\Qw)$ the set of all continuous affine functions $f$
on the compact convex set $\Qw$  such that $f=0,$ or $f(\tau)>0$ for all $\tau\in \Qw.$ 
Let $\LAff_+(\Qw)$ be  the set of those lowe-semi-continuous affine functions $f$
on $\Qw$  such that there are $f_n\in {\rm Aff}_+(\Qw)$ such that $f_n\nearrow f$ (point-wisely).
We allow $f$ has value $\infty.$} 

{\rm (4) In what follows, for any $0<\dt<1,$ denote by $f_\dt$ {{a function in}} $C([0, \infty))$ with $0\le f_\dt(t)\le 1$
for all $t\in [0,\infty),$ $f_\dt(t)=0$ if $t\in [0, \dt/2],$ $f_\dt(t)=1$ if $t\in [\dt, \infty)$ and $f_\dt(t)$ is linear 
on $(\dt/2, \dt).$}

{\rm Let $a, b\in (A\otimes {\cal K})_+.$ Define 
$$ d_\tau(a)=\lim_{n\to\infty}\tau(f_{1/n}(a))\rforal \tau\in {\widetilde{QT}}(A).$$
Note $f_{1/n}(a)\in {\rm Ped}(A)_+$ for each $n\in \N.$}

{\rm (5)  Let $A$ be an algebraically simple \CA\, with $QT(A)\not=\emptyset.$
 We say $A$ has strict comparison (of Blackdar, see \cite{B1}), if, for any $a, b\in (A\otimes {\cal K})_+,$  $d_\tau(a)<d_\tau(b)$ for all $\tau\in \Qw$ implies that 
$a\lesssim b$ in ${\rm Cu}(A).$ 

 (6)  $A$ is said to have finite radius of comparison, if there is $0<r<\infty$ such that, for any 
$a, b\in (A\otimes {\cal K})_+,$
$a\lesssim b,$ whenever 
$d_\tau(a)+r<d_\tau(b)$ for all $\tau\in \Qw$ (see \cite{Tomrc}). 

 If $A$ is $\sigma$-unital simple which is  not algebraically simple, we may pick a nonzero element
$e\in {\rm Ped}(A)_+$ and consider ${\rm Her}(a)\otimes {\cal K}\cong A\otimes {\cal K}.$ Then we say 
that $A$ has strict comparison (or has finite radius comparison), if ${\rm Her}(e)$ does.  It should be noted 
that this definition does not depend on the choice of $e.$}
Note that, in both cases above, we always assume that ${\widetilde{QT}}(A)\not=\{0\}.$
\end{NN}

\begin{NN}\label{q5}
{\rm In this section, we will consider the following question:
Suppose that  $d_\tau(b)$ is much larger than $d_\tau(a)$ 
for all $\tau\in \Qw.$ 
Does   it follow that $a\simle b?,$ or equivalently,
that $H_a$ is unitarily equivalent to a Hilbert $A$-submodule of $H_b?$
} 
\end{NN}

\begin{NN}

{\rm (7) Define a map $\Gamma: {\rm Cu}(A)\to \LAff_+(\Qw)$ by
$\Gamma([a])(\tau)=\widehat{[a]}(\tau)=d_\tau(a)$ for all $\tau\in \Qw.$ }

{\rm (8) It is proved by M. Rordam \cite{Ror04JS} that, for a unital finite simple \CA\, $A$,  if $A$ is ${\cal Z}$-stable then $A$ has stable rank one. Robert \cite{Rlz} showed that  any 
$\sigma$-unital simple stably projectionless ${\cal Z}$-stable \CA\, $A$ has 
almost stable rank one.   Recently, it is shown that, any $\sigma$-unital finite simple ${\cal Z}$-stable \CA\, 
has stable rank one (see \cite{FLL}).}

{\rm A part of the  Toms-Winter  conjecture states that, for a separable amenable simple \CA\, $A,$ $A$ 
is ${\cal Z}$-stable if and only if $A$ has strict comparison. 
It follows from \cite{Ror04JS}  and \cite{ERS} that if $A$ is ${\cal Z}$-stable, then $A$ has strict comparison and 
$\Gamma$ is surjective.   The remaining open question is whether a separable amenable 
simple \CA\, $A$ with strict comparison  is  always ${\cal Z}$-stable.
There are steady progresses to resolve the remaining problem (\cite{KR2}, \cite{S-2}, 
\cite{TWW-2}, \cite{Zw}, \cite{Th}, \cite{CETW}, \cite{LinGamma} and \cite{LinNZ}, for example)}.

{\rm Since strict comparison is a property for ${\rm Cu}(A),$  one may ask the question 
whether a separable amenable simple \CA\,$A$  whose ${\rm Cu}(A)$ behaves 
like a separable simple ${\cal Z}$-stable \CA\, is in fact ${\cal Z}$-stable. 
To be more precise, let us assume that $A$ is a separable simple \CA\, such that
${\rm Cu}(A)=V(A)\sqcup (\LAff_+(\widetilde{QT}(A))\setminus \{0\})={\rm CH}(A),$
where $V(A)$ is the {{sub-semigroup}}  of ${\rm Cu}(A)$ {{whose elements are}} represented by projections,
i.e., $A$ has strict comparison, $\Gamma$ is surjective and ${\rm Cu}(A)={\rm CH}(A).$
The question is whether  such $A$ is ${\cal Z}$-stable, if we also assume $A$ is amenable.
Theorem \ref{Tstr=1} shows that such \CA s always have stable rank one, by 
showing that these \CA s  have tracial approximate oscillation zero (see \ref{Domega}).  It should be also mentioned 
that 
if $A$ is a separable unital simple \CA\, with stable rank one,
then $\Gamma$ is surjective (see \cite{Th} and \cite{APRT}). }

\end{NN}

To state the next lemma, let us recall the definition of tracial approximate oscillation zero.

\begin{df}\label{Domega}
Let $A$ be a \CA\, with ${\widetilde{QT}}(A)\setminus \{0\}\not=\emptyset.$ 
Let $S\subset {\widetilde{QT}}(A)$
be a compact subset. 
Define, for each $a\in (A\otimes {\cal K})_+,$ 
\beq
\omega(a)|_S&=&\lim_{n\to\infty}\sup\{d_\tau(a)-\tau(f_{1/n}(a)): \tau\in S\}
\eneq
(see A1 of \cite{eglnkk0} and Definition 4.1 of \cite{FLosc}).  We will assume that $A$ is algebraically simple  and 
only consider the case that $S=\Qw,$ and in this case, we will write 
$\omega(a)$ instead of $\omega(a)|_{\Qw},$ in this paper.  It should be mentioned 
that $\omega(a)=0$ if and only if $d_\tau(a)$ is continuous (and finite) on $\Qw.$

In this case, (when $\|a\|_{_{2, \Qw}}<\infty,$ for example,
$a\in {\rm Ped}(A\otimes {\cal K})$), 
we write $\Omega^T(a)=\Omega^T(a)|_S=0,$ 
if there is a sequence $b_n\in {\rm Ped}(A\otimes {\cal K})\cap {\rm Her}(a)_+$ such that $\|b_n\|\le \|a\|,$
\beq
\lim_{n\to\infty}\omega(b_n)=0\andeqn \lim_{n\to\infty}\|a-b_n\|_{_{2, \Qw}}=0,
\eneq
where $\|x\|_{_{2, \Qw}}=\sup\{\tau(x^*x)^{1/2}: \tau\in \Qw\}$ (see Proposition 4.8 of \cite{FLosc}).
If $\omega(a)=0,$ then $\Omega^T(a)=0$ (see the paragraph after Definition 4.7 of \cite{FLosc}).
 
Even if $A$ is not algebraically simple (but $\sigma$-unital and simple), 
we may fix a  nonzero element $e\in {\rm Ped}(A)\cap  A_+$ (so that ${\rm Her}(e)\otimes {\cal K}\cong 
A\otimes {\cal K}$) and choose $S=\overline{QT({\rm Her}(e))}^w.$ 
Then, for any 
$a\in (A\otimes {\cal K})_+$ with $\|a\|_{_{2, S}}<\infty,$ 
we write $\Omega^T(a)=0$ if $\Omega^T(a)|_S=0$ 
(this does not depend on the choice of $e,$ see Proposition 4.9 of \cite{FLosc}).

A $\sigma$-unital simple \CA\, $A$ is said to have tracial approximate oscillation zero, if
$\Omega^T(a)=0$ for 
every positive element $a\in {\rm Ped}(A\otimes {\cal K}).$  
\end{df}

\begin{lem}\label{Lcont}
Let $A$ be a $\sigma$-untal algebraically  simple \CA\,  with $QT(A)\not=\emptyset,$ 
strict comparison and surjective $\Gamma.$ 
For any $a\in (A\otimes {\cal K})_+,$ there is $b\in (A\otimes {\cal K})_+,$ such 
that $\Gamma([a])=\Gamma([b])$ in ${\rm Cu}(A)$ such that
$\Omega^T(b)=0.$
\end{lem}

\begin{proof}
Recall that 
we assume that  $QT(A)\not=\emptyset.$ 
Let $g\in \LAff_+(\overline{QT(A)}^w)$ be such that $\widehat{[a]}=g.$
There is a sequence of increasing $g_n\in \Aff_+(QTA)$ such that
$g_n\nearrow g$ (pointwisely).

Put $f_1=g_1-(1/4)g_1\in {\rm Aff}_+(\Qw)$ and $\af_1=1/4.$  Then $g_2(\tau)>f_1(\tau)$
for all $\tau\in \Qw.$ Choose $0<\af_2<1/2^{2+1}$ such that
$(1-\af_2)g_2(\tau)>f_1(\tau)$ for all $\tau\in \Qw.$ Put $f_2=(1-\af_2)g_2.$ 
Suppose that $0<\af_i<1/2^{2+i}$ is chosen, $i=1,2,...,n,$ such 
that $(1-\af_{i+1})g_{i+1}(\tau)> (1-\af_i)g_i(\tau)$ for all $\tau\in \Qw,$ $i=1,2,...,n-1.$
Then $g_{n+1}(\tau)>(1-\af_n)g_n(\tau)$ for all $\tau\in \Qw.$ Choose 
$0<\af_{n+1}<1/2^{2+n}$ such that $(1-\af_{n+1})g_{n+1}(\tau)>(1-\af_n)g_n(\tau)$ 
for all $\tau\in \Qw.$ Define $f_{n+1}=(1-\af_{n+1})g_{n+1}\in {\rm Aff}_+(\Qw).$
By induction, we obtain a strictly increasing sequence $f_n\in {\rm Aff}_+(\Qw)$
such that $f_n\nearrow g.$

Define $h_1=f_1$ and $h_n=f_n-f_{n-1}$ for $n\ge 2.$ 
Then $h_n\in {\rm Aff}_+(\Qw)$ and $g=\sum_{n=1}^\infty h_n$
(converges point-wisely).
Since $\Gamma$ is surjective, there are $a_n\in (A\otimes {\cal K})_+$
such that $\widehat{[a_n]}(\tau)=d_\tau(a_n)=h_n,$ $n\in \N.$
We may assume that $a_ia_j=0$ if $i\not=j.$ 
Define $b\in (A\otimes {\cal K})_+$ by $b=\sum_{n=1}^\infty a_n/n.$

Then $\widehat{[b]}=g.$  In other words, $\Gamma([a])=\Gamma([b]).$

For each $m\in \N,$ since $\widehat{[a_n]}$ is  continuous on $\Qw,$ 
(by Lemma 4.5 of \cite{FLosc}), choose $e_{n,m}\in {\rm Her}(a_n)$ with 
$0\le e_{n,m}\le 1$ such that  $\{e_{n,m}\}_{m\in \N}$ forms an approximate 
identity (for ${\rm Her}(a_n)$),
\beq
d_\tau(a_n)-\tau(e_{n,m})<1/2^{n+m}\rforal \tau\in \Qw,
\andeqn \omega(e_{n,m})<1/2^{n+m}.
\eneq

Define $e_n=\sum_{j=1}^n e_{j, n}.$ Since $a_ia_j=0,$ if $i\not=j,$  {{we have  that}}
$0\le e_n\le 1,$ and $\{e_n\}$ forms an approximate identity for ${\rm Her}(b).$
By (2) of Proposition 4.4 of  \cite{FLosc}, we compute that
\beq
\omega(e_n)\le \sum_{j=1}^n \omega(e_{j, n})<\sum_{j=1}^n 1/2^{j+n}=1/2^n.
\eneq
Hence $\lim_{n\to\infty}\omega(e_n)=0$ and, by Proposition 5.7 of \cite{FLosc}, 
$\Omega^T(b)=0.$
\end{proof}

\begin{thm}
\label{Tstr=1}
Let $A$ be a $\sigma$-unital simple \CA\, with 
strict comparison and surjective $\Gamma$ (which is not purely infinite). 
Then the following are equivalent:

(1) for any $a, b\in (A\otimes {\cal K})_+,$ $[a]=[b]$ in ${\rm Cu}(A)$ if  and only if
$a \cuapprox b;$

(2) $A$ has stable rank one;

(3) ${\rm Cu}(A)={\rm CH}(A);$

(4) for any $a, b\in (A\otimes {\cal K})_+,$ $[a]=[b]$ in ${\rm Cu}(A)$ if  and only 
if Hilbert $A$-module $H_a$ and $H_b$ are unitarily equivalent.

\end{thm}

\begin{proof}

(2) $\Rightarrow$ (3) follows from Theorem 3 of \cite{CEI}.

(3) $\Leftrightarrow$ (4)  follows from the definition.

(1) $\Leftrightarrow$ (4) follows from Proposition \ref{Prelation}. 

(3) $\Rightarrow$ (2):   We will show that $A$ has tracial approximate oscillation zero (see Definition  5.1
of \cite{FLosc}).

Let $a\in ({\rm Ped}(A\otimes {\cal K}))$  with $0\le a\le 1.$
Suppose that $\Gamma([a])$ 
is continuous. Then $\omega([a])=0.$  Hence 
$\Omega^T(a)=0.$
Now suppose that $\Gamma([a])$ is not continuous.
In particular, $[a]$ cannot be represented by a projection. 
 By Lemma \ref{Lcont}, 
there exists $b\in (A\otimes {\cal K})_+$ such that
$\Omega^T(b)=0$ 
and $\Gamma([a])=\Gamma([b]).$  
Since we now assume that $\Gamma([a])$ is not continuous neither is $\Gamma([b]).$
Hence both $[a]$ and $[b]$ are not represented by projections. 
As we assume that $A$ has strict comparison, this implies that $[a]=[b].$ 

Since (3) holds, 
$\overline{a(H_A)}$ is unitarily equivalent to $\overline{b(H_A)}.$
We have shown (3) $\Leftrightarrow$ (4) and (1) $\Leftrightarrow$ (4). 
By Proposition \ref{Prelation}, there is a partial isometry $v\in (A\otimes {\cal K})^{**}$
such that $v^*cv\in \overline{a(A\otimes {\cal K})a}$ for all $c\in \overline{b(A\otimes {\cal K})b}$ 
and $v^*bv$ is a strict positive element of $\overline{a(A\otimes {\cal K})a}.$ It follows 
that $\Omega^T(a)=0.$ Since this holds for every $a\in {\rm Ped}(A\otimes {\cal K}),$ 
$A$ has tracial approximate oscillation zero. By Theorem 9.4 of \cite{FLosc}, $A$ has stable rank one,
i.e., (2) holds.
%
\end{proof}



Let us end this section with the following partial answer to the question in \ref{q5}. 
Note that $d_\tau(b)$ is as large as one can  possibly have.

\begin{thm}
Let $A$ be a $\sigma$-unital simple \CA\, with finite radius of comparison. 
Suppose that $a, b\in (A\otimes {\cal K})_+$ such that $d_\tau(b)=\infty$ for all $\tau\in {\widetilde{QT}}(A)\setminus \{0\}.$ 
Then 

(1) $a\simle b$ and $H_a$ is unitarily equivalent to an orthogonal summand 
of $H_b,$

(2) $H_b\cong H_A$ as Hilbert $A$-module, and 

(3) ${\rm Her}(b)\cong A\otimes {\cal K}.$
\end{thm}

\begin{proof}
Put $B=\overline{b(A\otimes {\cal K})b}.$ 
Let us show that $B$ is stable. We will use the characterization of stable \CA s of \cite{HRm}.
Moreover, we use the following fact: If $A$ has finite radius of comparison and $c, b\in (A\otimes {\cal K})_+$ such that  $d_\tau(d)<\infty$ and 
$d_\tau(b)=\infty$ for all $\tau\in {\widetilde{QT}}(A)\setminus \{0\},$ then $d\lesssim b.$  

Let $c\in B_+$ such that there exists $e\in B_+$ such that $ec=c.$ 
Working in the commutative \SCA\, generated by $c$ and $e,$ we conclude that, if $0<\dt<1/2,$
$f_\dt(e)c=c.$ 

Let $b_1:=(1-f_{1/8}(e))^{1/2}b(1-f_{1/8}(e))^{1/2}.$ 
Then 
\beq
b_1c=cb_1=0.
\eneq
 Note that $f_{1/16}(e)\in {\rm Ped}(A\otimes {\cal K})_+.$ Therefore 
 $$
 d_\tau(f_{1/8}(e))\le \tau(f_{1/16}(e))<\infty\rforal \tau\in {\widetilde{QT}}(A).
 $$
We also have 
\beq
b
\lesssim
b^{1/2} (1-f_{1/8}(e))b^{1/2}\oplus b^{1/2}f_{1/8}(e)b^{1/2}.
\eneq
Since $b^{1/2}f_{1/8}(e) b^{1/2}\lesssim f_{1/8}(e),$ then $d_\tau(b^{1/2}f_{1/8}(e)b^{1/2})<\infty$
for all $\tau\in {\widetilde{QT}}(A).$ 
Hence 
\beq
d_\tau(b_1)=d_\tau(b^{1/2} (1-f_{1/8}(e))b^{1/2})=\infty\rforal \tau\in \widetilde{QT}(A)\setminus \{0\}.
\eneq
Since $A$ has finite radius of comparison, as mentioned above, we have 
that
\beq
f_{1/8}(e) \lesssim b_1. 
\eneq
There is, by Lemma 2.2 of \cite{R1},  $x\in A\otimes {\cal K}$ such that
\beq
x^*x=f_{1/16}(f_{1/8}(e))\andeqn  xx^*\in {\rm Her}(b_1).
\eneq
Let $x=v|x|$ be the polar decomposition of $x$ in $A^{**}.$ 
Then $\phi: {\rm Her}(f_{1/6}(f_{1/8}(e)))\to {\rm Her}(b_1)$ defined by
$\phi(d)=vdv^*$ for all $d\in {\rm Her}(f_{1/6}(f_{1/8}(e)))$ is \hm. 
Note that $f_{1/6}(f_{1/8}(e))\le f_{1/4}(e)$ and $f_{1/4}(e)c=c.$
Put $y=vc^{1/2}.$ Then $y\in A\otimes {\cal K}$ and $y^*y=c$ and $yy^*\in {\rm Her}(b_1).$
Hence 
\beq
c=yy^* \andeqn c\perp yy^*.
\eneq 
Since $c$ is chosen arbitrarily in $B_+$ with the property that there is $e\in B_+$ such that $ec=c,$ 
it follows from Theorem 2.1 of \cite{HRm} that $B$ is stable.  By Brown's stable isomorphism theorem (\cite{Brst}), $B\cong A\otimes {\cal K}.$ 
This proves (3). 

Since $B$ is stable, there is  a sequence of {{mutually orthogonal}} nonzero 
elements  $b_{0,n}\in B_+$ ($n\in \N$) such that
$b_{0,1}\cuapprox b_{0,n}$ for all $n\in \N$ and $b_0=\sum_{n=1}^\infty b_{0,n}/n\in B$ is a 
strictly positive element.  We have $H_{b_{0,n}}\cong H_{b_{0, m}}$ for all $n, m\in \N.$
It follows that $H_b=H_{b_0}\cong l^2(H_{b_0})$ as Hilbert $A$-modules.
By Proposition 7.4 of \cite{Lan}, $l^2(H_{b_0})\cong l^2(A)=H_A$ as Hilbert $A$-module.   This proves (2).

For (1), since we have shown that $H_b\cong H_A,$  we may apply Kasparov's absorbing theorem (\cite{K}). 
\end{proof}

\section{Projective Hilbert Modules}

The main result of this section is Theorem \ref{Tproj} which states that, for separable \CA\, $A,$
 every countably generated Hilbert $A$-module 
is projective. 

Note that in the following statement, we use the fact that $B(H)=LM(H)$ 
(see Theorem 1.5 of \cite{Lninj}).

\begin{lem}\label{Ladd}
Let $A$ be a  \CA\, and $H$ be Hilbert $A$-modules.
Suppose that $T\in B(H) $ is a bounded module map 
such that $T(H)$ is a Hilbert $A$-submodule of $H$ and $L\in LM(K(H))$ 
the corresponding left multiplier. 
Then $LK(H)=K(T(H)).$ 
\end{lem}

\begin{proof}
Let $H_1=T(H)$ and  $F(H)$ be the linear span of bounded module maps of the form 
$\theta_{x, y}$ for $x, y\in H.$ 
Note that $Lb(x)=Tb(x)$ for all $b\in K(H)$ and $x\in H.$ 

Let $\xi,\zeta\in H_1\subset H.$ 
Since $T:H\to H_1$ is surjective, {{there is $x\in H$ such that 
$T(x)=\xi.$ Then $T\circ \theta_{x, \zeta}=\theta_{\xi, \zeta}.$
This implies that}} $LF(H)=F(H_1).$
Let $H_0={\rm ker} T$ be  as a Hilbert submodule of $H.$ Then $K(H_0)\subset K(H)$
is a hereditary \SCA\, of $K(H)$ (see Lemma 2.13 of \cite{Lninj}). Let $p$ be the 
open projection in $K(H)^{**}$  corresponding to $K(H_0).$ Then (working in $H^\sim$ if necessary) $Lp=0$ and 
$Lb=L(1-p)b$ for all $b\in K(H).$ We identify $H/H_0$ with $(1-p)H$ ($\subset H^\sim$). Let $\tilde T: (1-p)H\to H_1$ be 
the bounded module map induced by $T$ which has a bounded inverse $\tilde T^{-1}$
as $T$ is surjective. Note that $\tilde T^{-1}(T(z))=(1-p)z$ for all $z\in H.$

To show that $T(K(H))=K(H_1),$ fix $S\in K(H_1).$ Let $S_n\in F(H_1)$ be such that\\
 $\lim_{n\to\infty}\|S_n-S\|=0.$
Since $LF(H)=F(H_1),$ we may choose $F_n\in F(H)$ such that $LF_n=S_n,$ $n\in \N.$
Then, for any $x\in H,$ $\tilde T^{-1}(LF_n(x))=(1-p)F_n(x).$ 

We claim that $\{(1-p)F_n\}$ converges in norm to an element of the form $(1-p)b$ for some 
$b\in K(H).$
Note that, for any $n,m\in \N,$ and any $x\in H,$
\beq\nonumber
\|(1-p)F_n(x)-(1-p)F_m(x)\|=\|\tilde T^{-1}(LF_n(x)-LF_m(x))\|
\le \|\tilde T^{-1}\|\|S_n-S_m\|\|x\|.
\eneq
This implies  that $\{(1-p)F_n\}$ is Cauchy in norm and it must converges to an 
element of the form $(1-p)b$ for some $b\in K(H),$ as $(1-p)K(H)$ is closed. 
It then follows that, for any $b\in K(H),$
\beq\nonumber
Lb=L(1-p)b=\lim_{n\to\infty}L(1-p)F_n=\lim_{n\to\infty}S_n= S.
\eneq
This shows that $L K(H)=K(H_1).$
\end{proof}

\begin{df}\label{Dtildes-1}
Let $H$ and $H_1$ be Hilbert $A$-modules and 
$T: H_1\to H$ be a surjective bounded module map. Then $T$ induces an invertible bounded map 
${\tilde T}: H_1/{\rm ker}T\to H$ from  Banach $A$-modules $H_1/{\rm ker} T$ onto  $H.$
  Denote by ${\tilde T}^{-1}: H\to H_1/{\rm ker}T$ the inverse (which is also bounded).
\end{df}

\begin{thm}\label{BLT1}
Let $A$ be a $\sigma$-unital \CA. 

For any  countably generated  Hilbert $A$-modules $H_1,H_2, H_3,$ any bounded module maps 
$T_1: H_1\to H_3$ and $T_2: H_2\to H_3.$ Suppose that
$T_1$ is surjective,
then there is a bounded module map $T_3: H_2\to H_1$ such 
that
$$
T_1\circ T_3=T_2.
$$
Moreover, $\|T_3\|= 
\|{\tilde T}_1^{-1}\circ T_2\|.$

\end{thm}

\begin{proof}
Without loss of generality, we may assume that $\|T_2\|=1.$ 
Let $H_4=H_1\oplus H_A,$ $H_5=H_2\oplus H_A$ and $H_6=H_3\oplus H_A.$
Let $T_4=T_1\oplus {\rm id}_{H_A}: H_4\to H_6$  and $T_5=T_2\oplus {\rm id}_{H_A}: H_5\to H_6.$ 
Suppose $T_1$ is surjective. Then $T_4$ is surjective. 
Suppose that there is a bounded module map $T': H_5\to H_4$ such that
$$
T_4\circ T'=T_5\andeqn \|T'\|= 
\|{\tilde T}_4^{-1}\circ T_5\|
$$
One computes  that 
$$
\|{\tilde T}_4^{-1}\circ T_5\|=
\|{\tilde T}_1^{-1}\circ T_2\|.
$$
One also has that 
$$
T_4\circ T'|_{H_2}=T_5|_{H_2}=T_2.
$$
Since ${\rm range}(T_2)\subset H_3,$ $T_4(T'(H_2))=T_5(H_2)\subset H_3.$ 
Let $P_1: H_4\to H_1,$ $P_{H_A}: H_4\to H_A$  and $P_3: H_6\to H_3$ be the orthogonal projections. Then 
$P_3\circ T_4\circ T'|_{H_2}=T_2.$  Write $T'=P_1T'+(1_{H_4}-P_1)T'.$
Note that $1_{H_4}-P_1=P_{H_A}$ and $P_3T_4(1_{H_4}-P_1)=0.$
Then 
\beq\label{BLT1-10}
P_3\circ T_4\circ T'|_{H_2}&=&P_3\circ T_4\circ P_1\circ T'|_{H_2}+P_3\circ T_4\circ (1_{H_4}-P_1)\circ T'|_{H_2}\\
&=&P_3\circ T_4\circ P_1\circ T'|_{H_2}+0\\\label{BLT1-10+}
&=&
T_1\circ P_1\circ T'|_{H_2}.
\eneq
Define $T=P_1\circ T'|_{H_2}.$  Then, by \eqref{BLT1-10+},
\beq\label{BLT1-11}
T_1\circ T=T_2\andeqn \|T\|=\|P_1\circ T'|_{H_2}\|=\|\tilde T_1^{-1}\circ T_2\|.
\eneq	

Therefore, without loss of generality, we may assume that 
$H_1\cong H_2\cong H_3=H_A.$  Put $B=A\otimes {\cal K}.$ $B$ is a $\sigma$-unital \CA. 
We view $T_1$ as a bounded module map from $H_A$ onto $H_A.$ 
Let $H_0={\rm ker} \,T_1.$ Then $H_0$ is a Hilbert  $A$-submodule of $H_A$ and 
$K(H_0)\subset K(H_A)=B$ is a hereditary \SCA\, (see Lemma 2.13 of \cite{Lninj}). Let $p\in B^{**}$ be the open projection 
corresponding to the hereditary \SCA\, $K(H_0).$ Let $L_1\in LM(B)$ be given by the bounded module 
map $T_1$ (see Theorem 1.5 of \cite{Lnbd}).  Note that $L_1b(z)=T_1(b(z))$ for all $z\in H_A$ and $b\in B.$
For any $a\in K(H_0)\subset B,$ $L_1a=0.$  It follows that $L_1p=0$ and 
$L_1(1-p)=L_1.$

Let $J=\overline{K(H_0)B}$ be the right ideal of $K(H_A)=B.$ Consider the quotient Banach $B$-module 
$B/J.$  Note that the module map ${\bar b}\mapsto (1-p)b$ is an isometry. So we may 
identify $B/J$ with $(1-p)B.$  Define $\tilde L_1: B/J\to B$ by $\tilde L_1((1-p)b)=L_1b$
for $b\in B.$  Viewing $B$ as a Hilbert $B$-module,  one has $K(B)=B.$ 
By Lemma \ref{Ladd},  viewing $B$ as Hilbert $B$-module,
$L_1B=L_1K(B)=K(B)=B.$ It follows that 
$\tilde L_1$ is also a surjective map and hence has bounded inverse as a bounded linear map.
Denote its inverse as $\tilde L_1^{-1}.$  Note that ${\tilde L}_1^{-1}L_1b=(1-p)b$\, for $b\in B$
(recall that $L_1$ is surjective). 
 In particular,  ${\tilde L}_1^{-1}$ is a bounded module map 
from $K(H_A)=B$ to $(1-p)B.$

We may also use the identification $H_A/H_0=(1-p)H_A.$   Now $T_1$ induces a bounded module 
map ${\tilde T}_1$ from $(1-p)H_A$ onto $H_A$ which has a bounded inverse ${\tilde T}_1^{-1}.$ 
Let $L_2\in LM(B)$ be given by the bounded module map $T_2.$ 
Then the map $L_2: B\to B$ given by $L_2(b)=L_2b$ for all $b\in B$ is a bounded $B$-module map 
from $B$ into $B=K(H_A).$ 

Consider the bounded $B$-module map ${\tilde L}_1^{-1}\circ L_2: B\to B/J.$ 
Then, by 3.11 of \cite{Br2}, there exists a bounded module map $L: B\to B$ such that
\beq
(1-p)L={\tilde L}_1^{-1}\circ L_2\andeqn \|L\|= \|{\tilde L}_1^{-1}\circ L_2\|.
\eneq

It follows that ${\tilde L_1}(1-p)L=L_2.$  Recall that $\tilde L_1(1-p)=L_1.$  Therefore
\beq\label{BLT1-12}
L_1\circ L=L_2.
\eneq
We identify $L$ with a left multiplier in $LM(B)$ (see Theorem 1.5 of \cite{Lnbd}). 
Let $T: H_1=H_A\to H_A=H_3$ be the bounded module map 
given by $L$ (see again Theorem 1.5 of \cite{Lnbd}). Then 
we obtain 
that
\beq\label{BLT1-13}
T_1\circ T=T_2\andeqn \|T\|=\|\tilde T_1^{-1}\circ T_2\|.
\eneq

The commutative diagrams on $B$ level may be illustrated as follows (with $\pi: B\to B/J$ the quotient map
given by $(1-p)$)
$$
\begin{array}{ccccc}
\hspace{-0.1in}B &&\\
\vspace{0.1in}
\downarrow_{L_2}& \searrow^L&\\
\hspace{-0.1in}B&\hspace{-0.1in}\twoheadleftarrow^{L_1} &B\\
\vspace{0.1in}\hspace{-0.2in}^{^{\widetilde L_1^{-1}}}\updownarrow_{_{\widetilde L_1}}& \swarrow_{\pi}&\\
B/J &&
\end{array} \,\,\,\,{\rm and}\hspace{0.6in} \begin{array}{ccccc}
   B&\longrightarrow^{L}&B\\
   \vspace{0.1in}\hspace{-0.3in}^{^{\widetilde L_1^{-1}L_2}}\downarrow 
  & \swarrow_\pi&\\
   B/J&&\end{array}.
$$


\end{proof}

\begin{cor}\label{Projm}
Let $A$ be a $\sigma$-unital \CA\, and let 
$$
0\to H_1\to^{\iota} H_2\to^{s} H_3\to 0
$$
be a short exact sequence  of countably generated Hilbert $A$-modules. Then 
it splits.  Moreover, the splitting map $j: H_3\to H_2$ has $\|j\|=\|{\tilde s}^{-1}\|.$
\end{cor}

\begin{proof}
Consider the diagram:
$$
\begin{array}{ccccc}
&& H_2 &\\
& & \hspace{0.1in}\downarrow_{s} &\\
H_3 & \stackrel{{\rm id}_{_{H_3}}}{\twoheadrightarrow}
& H_3 
\end{array}
$$
By Theorem \ref{BLT1}, there is a bounded module map $T: H_3\to H_2$ such that
$$
s\circ T={\rm id}_{H_3}\andeqn \|T\|=\|\tilde s^{-1}\|.
$$
\end{proof}



We need  the following lemma which the authors could not locate a
reference.

\begin{lem}\label{openmap}
Let $X$ be a Banach space and let $H$ be a separable Banach space.
Suppose that $T: X\to H$ is a surjective bounded linear map.  Then
there is a separable subspace $Y\subset X$ such that $TX=H.$
\end{lem}

\begin{proof}
Note that the Open Mapping Theorem applies here. From the open
mapping theorem (or a proof of it), there is $\dt>0$ for which
$T(B(0, a))$ is dense in $O(0, a\dt)$ for any $a>0,$  where $B(0,
a)=\{x\in X: \|x\|\le a\}$ and $O(0, b)=\{h\in H: \|h\|<b\}.$ For
each rational number $r>0,$ since $H$ is separable, one may find a
countable set $E_r\subset B(0, r)$ such that $T(E_r)$ is dense in
$O(0,r\dt).$ Let $Y$ be the closed subspace generated by $\cup_{r\in
\Q_+} E_r.$

Let $d=\dt/2$ and let $y_0\in O(0, d).$ Then $T(Y\cap B(0, 1/2))$ is
dense in $O(0,d).$ Choose $\xi_1\in Y\cap B(0,1/2)$ such that
\beq\label{Opm-1}
\|y_0-T\xi_1\|<\dt/2^2.
\eneq
In particular,
\beq\label{Opm-2}
y_1=y_0-T\xi_1\in O(0, \dt/2^2).
\eneq
Since $T(Y\cap B(0, 1/2^2))$ is dense in $O(0, \dt/2^2),$ one
obtains $\xi_2\in Y\cap B(0, 1/2^2)$ such that
\beq\label{Opm-3}
\|y_1-T\xi_2\|<\dt/2^3.
\eneq
In other words,
\beq\label{Opm-4}
y_2=y_1-T\xi_2=y_0-(T\xi_1+T\xi_2)\in O(0, \dt/2^3).
\eneq
Continuing this process, one obtains  a sequence of elements
$\{\xi_n\}\subset Y$ for which $\xi_n\in B(0, 1/2^n)$ and
\beq\label{Opm-5}
\|y_0-(T\xi_1+T\xi_2+\cdots +T\xi_n)\|<\dt/2^{n+1},\,\,\,n=1,2,....
\eneq
Define $\xi_0=\sum_{n=1}^{\infty}\xi_n.$ Note that the sum converges
in norm and therefore $\xi_0\in Y.$ By the continuity of $T,$
\beq\label{Opm-6}
T\xi_0=y_0.
\eneq
This implies that $T(Y)\supset O(0, d).$  It follows that $T(Y)=H.$
\end{proof}

\begin{thm}\label{Tproj}
Let $A$ be a separable  \CA. 
Then every countably generated Hilbert $A$-module is projective 
in the following sense:

For any    Hilbert $A$-modules $H_1,H_2, H_3,$ any bounded module maps 
$T_1: H_1\to H_3$ and $T_2: H_2\to H_3.$ Suppose that
$T_1$ is surjective and $H_2$ is countably generated,  then there is a bounded module map $T_3: H_2\to H_1$ such 
that
$$
T_1\circ T_3=T_2.
$$
Moreover, $\|T_3\|= \|{\tilde T}_1^{-1}\circ T_2\|.$
\end{thm}

The statement above may be illustrated as follows:

\begin{center}
\begin{large}
\begin{equation}
\begin{tikzcd}
H_2 \ar[d, "T_2"'] \ar[dr, dashed, "T_3"]\\
H_3 & H_1 \ar[l, two heads, "T_1"]
\end{tikzcd}
\end{equation}
\end{large}
\end{center}

\begin{proof}
Let $H_3'$ be the Hilbert $A$-module generated by $T_2(H_2).$ Since $H_2$ is countably generated, 
so is $H_3'.$ Since $A$ is separable, $H_3'$ is separable. Let $H_1'=T^{-1}(H_3'),$
the pre-image of  $H_3'$ under $T_1.$
So $T_1|_{H_1'}: H_1'\to H_3'$ is surjective. Then, by Lemma \ref{openmap},  there is a separable Banach subspace of $S\subset H_1'$
such that $T_1(S)=H_3'.$  Let $\{x_n\}$ be a dense subset of $S.$
Define $H_1''$ to be the closure of ${\rm span}\{x_na: a\in A, n=1,2,....\}.$ Then $H_1''$ is a countably generated 
Hilbert $A$-module. Moreover, $S\subset H_1''\subset  H_1'\subset H_1.$ 
But  $T_1(S)=H_3'.$  Therefore  $T_1(H_1'')=H_3'.$
We have the following diagram:
\begin{center}
\begin{large}
\begin{equation}
\begin{tikzcd}
H_2 \ar[d, "T_2"'] 
\\
H_3'' & H_1'' \ar[l, two heads, "T_1"]
\end{tikzcd}
\end{equation}
\end{large}
\end{center}                     
Now $H_1', $ $H_2$ and $H_3''$ are all countably generated.
By \ref{BLT1}, there exists $T_3: H_2\to H_1''\subset H_1$ such that
\beq\label{Tproj-1}
T_1\circ T_3=T_2\andeqn \|T_3\|=\|{\tilde T}_1^{-1}\circ T_2\|.
\eneq

\end{proof}

\begin{cor}\label{Cproj}
Let $A$ be a separable  \CA\, and let 
$$
0\to H_1\to^{\iota} H_2\to^{s} H_3\to 0
$$
be a short exact sequence  of Hilbert $A$-modules.  Suppose that $H_3$ is countably generated.
Then  the short exact sequence 
splits.  Moreover, the splitting map $j: H_3\to H_2$ has $\|j\|=\|{\tilde s}^{-1}\|.$
\end{cor}

\begin{cor}\label{Cfunctor}
Let $A$ be a separable \CA.

(1)  Suppose that  $H$ is a countably generated Hilbert $A$-module.
Then, for any short exact sequence of Hilbert $A$-modules
$$
0\rightarrow H_0\stackrel{j}{\rightarrow} H_1\stackrel{s}{\rightarrow} H_2\rightarrow 0,
$$
one has the following short exact sequence:
$$
0\to B(H, H_0)\stackrel{j_*}{\rightarrow} B(H, H_1)  \stackrel{s_*}{\rightarrow}B(H,H_2)\rightarrow 0,
$$
where $j_*(\phi)=j\circ \phi$ for $\phi\in B(H, H_0)$ and 
$s_*(\psi)=s\circ \psi$ for $\psi\in B(H, H_1).$

(2) Suppose that $H_2$ is countably generated. Then, 
for any Hilbert $A$-module $H,$ and any short exact sequence
of Hilbert $A$-modules:
$$
0\rightarrow H_0\stackrel{j}{\rightarrow} H_1\stackrel{s}{\rightarrow} H_2\rightarrow 0,
$$
One has the splitting short exact sequence 
\beq\label{2rdshort}
0\to B(H, H_0)\stackrel{j_*}{\rightarrow} B(H, H_1)  \stackrel{s_*}{\rightarrow}B(H,H_2)\rightarrow 0,
\eneq
\end{cor}

\begin{proof}
 
 For (1), let us only show that $s_*$ is surjective. 
 Let $\psi\in B(H, H_2).$
Since $s: H_1\to H_2$ is surjective, by Theorem \ref{Tproj}, $H$ is projective.
There is $\tilde \psi\in B(H, H_1)$ such that $s\circ \tilde \psi=\psi.$ This implies that 
$s_*$ is surjective. 

For (2), by  Corollary \ref{Cproj},  there exists $\psi: H_2\to H_1$ such that
$
s\circ \psi={\rm id}_{H_2}.
$
For each $T\in B(H, H_2),$ define $\psi\circ T: H\to H_1.$ Then $s\circ (\psi\circ T)=T.$
Hence $s_*$ is surjective.  This implies that \eqref{2rdshort} is a short exact.
Moreover, map $\psi$ above also gives a splitting map $\psi_*: B(H, H_2)\to B(H, H_1).$
\end{proof}

\begin{df}\label{Dflat}
A finitely generated free Hilbert $A$-module is a Hilbert $A$-module 
with the form $H=e_1A\oplus e_2A\oplus \cdots \oplus e_nA,$ where $\la e_i, e_i\ra=p_i$ 
is a projection in $A,$ $i=1,2,...,n.$ Such a module is always self-dual (see, for example, Proposition \ref{Probert}). 

There is a notion of torsion free modules. Every Hilbert $A$-module is torsion free in the following sense.
Let $x\in H$ be a non-zero element  and $xa=0$ for some $a\in A.$ Then $a$ must be a left zero divisor. 
In fact,  if $xa=0,$ then $a\la x,x\ra a=0,$ or $\la x, x\ra^{1/2} a=0.$ Then $a$ must be a left zero divisor. 

One may define a Hilbert $A$-module $H$ to be {\it flat}, if, for any bounded module 
map $T: F\to H,$ where $F$ is a finitely generated free Hilbert $A$-module, 
there exists a free Hilbert $A$-module $G$ and a bounded module map 
$\phi: F\to G$ and a bounded module map $\psi: G\to H$ such that
$\phi({\rm ker}T)=0$ and $\psi\circ \phi=T$ as described as the following commutative diagram:

$$\begin{array}{ccccc}
   &&&G\\
   &&\nearrow_\phi&&\hspace{-0.1in} \searrow^\psi\\
   {\rm ker} T\hookrightarrow &\hspace{0.3in} F &&\stackrel{T}{\longrightarrow} & \hspace{0.2in}H\\
 \end{array}
\hspace{0.5in} \phi({\rm ker}T)=0
$$
This is equivalent to say the following:
for any $x_1, x_2,...,x_m\in H,$ if there are $a_i\in A,$ $1\le i\le m,$ 
such that   $\sum_{i=1}^m x_i a_i=0,$ 
there must be some integer $n,$
$y_j\in H,$ $1\le j\le m$ and $b_{i,j}\in A,$ $1\le j\le m,\, 1\le i\le n,$   such that 
$\sum_{i=1}^m a_ib_{i,j}=0$ and $x_i=\sum_{j=1}^n a_{i,j} y_j,$ $1\le i\le m.$
\end{df}

\begin{prop}\label{Tfalt}
Let $A$ be a $\sigma$-unital \CA\, and $H$ be any Hilbert $A$-module.
Suppose that $F$ is a self-dual Hilbert $A$-module and $T\in B(F, H).$ 
Then there are bounded module maps  $\phi: F\to F$ and $\psi: F\to H$ such that
$\psi\circ \phi=T$ and  $\phi|_{{\rm ker}T}=0.$ 
In particular, every Hilbert $A$-module is flat.
\end{prop}

\begin{proof}
Since $F$ is self-dual, 
 Therefore $T^*$ maps $H$ into $F$ (instead  into $F^\sharp$), or
 $T^*\in B(H, F).$  In other words, $T\in L(F, H).$
 Put $H_1=F\oplus H.$ Define $S\in B(H_1)$ by
 $$S(f\oplus h)=0\oplus T(f)\rforal \in F\andeqn h\in H.$$ Then $S^*\in B(H_1).$
 Let $S=V|S|$ be the polar decomposition of $S$ in $L(H)^{**}.$ 
 Then $V|S|^{1/2}, |S|^{1/2}\in L(H_1).$ 
 In other words, $|T|^{1/2}\in L(F)$ and $V|T|^{1/2}\in L(F, H).$ 
 
 Now define $\phi: F\to F$ by $\phi=|T|^{1/2}$ and $\psi: F\to H$ by $\psi=V|T|^{1/2}.$
 Then $\phi|_{{\rm ker} T}=0$ and $\psi\circ \phi=V|T|=T.$ 
 
 If $F$ is a finitely generated free Hilbert $A$-module, then $F$ is self-dual. So the above applies.
 Consequently $H$ is flat. 

\end{proof}

\section{Sequential approximation}
%
Let $A$ be a \CA\, and $H_0\subset H$ be  Hilbert $A$-modules. Suppose that 
$\phi: H_0\to A$ is a bounded module map.  Consider the Hahn-Banach type of extension 
question:
 whether there 
is a bounded module map $\psi: H\to A$ such that $\psi|_{H_0}=\phi$ (and 
$\|\psi\|=\|\phi\|$).  If $H_0\subset H$ is an arbitrary pair of Hilbert $A$-modules.
This is to ask whether $A$ is an injective Hilbert $A$-module with bounded module maps
as morphisms.
By Theorem 3.8 of \cite{Lninj}, if $A$ is a monotone complete \CA, then 
$A$ is an injective Hilbert $A$-module. Conversely, 
if $A$ is  an injective Hilbert $A$-module, then $A$ must be an $AW^*$-algebra
(see Theorem 3.14 of \cite{Lninj}).    In fact, injective Hilbert $A$-modules are rare. It may never 
happen when  $A$ is a separable but infinite dimensional  \CA\, (see \cite{Lninj} for some further discussion).

In what follows, when $x, y\in H$ and $\ep>0,$  we may write $x\approx_\ep y$
if $\|x-y\|<\ep.$

Let $H$ be a Hilbert $A$-module and $H^\sharp=B(H, A),$ the Banach $A$-module 
of all bounded $A$-module maps from $H$ to $A.$  Recall that  $H^\sharp\not=H$ in general (see 
the end of \ref{Hild}).
However, let us include the following approximation result.

\begin{thm}\label{Lapp-1}
Let $H$ be a Hilbert $A$-module and $\phi: H\to A$ be a bounded module map.
 Then, for any finite subset ${\cal F}\subset H,$ and any $\ep>0,$  there exists $x\in H$ such that
 \beq
 \|\phi(\xi)-\la x, \xi\ra\|<\ep\tforal \xi\in {\cal F}\andeqn \|x\|\le \|\phi\|.
 \eneq
  \end{thm}

\begin{proof}
We may assume that $\|\phi\|\not=0.$
Let  $H_1=H\oplus A$ and $P_1, P_A\in L(H_1)$  be the projections from $H_1$  onto $H$ and onto $A,$ 
respectively.
Define $T\in B(H_1)$ by $T=\phi\circ P_1.$ 
Note that, for $x\in H,$ $T(x)=\phi(x)$ and $\|T\|=\|\phi\|.$

Let ${\cal F}=\{\xi_1, \xi_2,...,\xi_k\}\subset H$ and $\ep>0.$
Without loss of generality, we may assume that $0\not=\|\xi_i\|\le 1,$ $1\le i\le k.$
Let $\{E_\lambda\}$ be an approximate identity for $K(H).$ 
By Lemma \ref{Lappunit}, there exists $\lambda_0$ such that, for all $\lambda\ge \lambda_0,$ 
\beq\label{Lapp-1n-1}
\|E_{\lambda}(x)-x\|<\ep/4(\|\phi\|+1)\rforal x\in {\cal F}
\eneq

Fix $\lambda\ge \lambda_0.$
Applying Theorem 1.5 of \cite{Lnbd}, $B(H_1)=LM(K(H_1)).$ It follows that $TE_\lambda P_1\in K(H_1).$
Note also that $\|TE_\lambda P_1\|\not=0.$ 
Therefore, there are $y_1', y_2',...,y_m', z_1,z_2,...,z_m\in H_1$ such that
\beq\nonumber
\|TE_\lambda P_1-\sum_{j=1}^m \theta_{y_j', z_j}\|<\ep/8,\,\,1\le j\le m.
\eneq 
Then
\beq
\|\sum_{j=1}^m \theta_{y_j',z_j}\|\le \|TE_\lambda P_1\|+\ep/8.
\eneq
Put $\bt={\|TE_\lambda P_1\|\over{\|TE_\lambda P_1\|+\ep/8}}$ and $y_j=\bt y_j',$ 
$1\le j\le m.$
Then
\beq\label{16-10}
\sum_{j=1}^m \theta_{y_j, z_j}=\bt \sum_{j=1}^m \theta_{y_j', z_j}\andeqn
 \|\sum_{j=1}^m \theta_{y_j, z_j}\|\le 
\|TE_\lambda P_1\|\le \|\phi\|.
\eneq
 Moreover,
\beq
\|TE_\lambda P_1-\sum_{j=1}^m \theta_{y_j, z_j}\|\le \|TE_\lambda P_1-\sum_{j=1}^m \theta_{y_j',z_j'}\|+(1-\bt)\|\sum_{j=1}^m \theta_{y_j',z_j'}\|\\\label{Lapp-1-5}
<\ep/8+(1-\bt)(\|TE_\lambda P_1\|+\ep/8)=\ep/8+\ep/8=\ep/4.
\eneq
Denote $L=\sum_{j=1}^m \theta_{y_j, z_j}.$ Then, by \eqref{16-10}, $\|L\|\le \|\phi\|.$ 
Note that  $P_1\in L(H_1).$ Therefore, for all $z\in H,$  
$\theta_{y_j, z_j}P_1(z)=y_j\la z_j, P_1z\ra=y_j\la P_1z_j, z\ra=\theta_{y_j, P_1z_j}(z),$ $1\le j\le m.$
By replacing $z_j$ by $P_1z_j,$ we may assume that $z_j\in H.$ 
Recall $P_ATE_\lambda P_1=TE_\lambda P_1.$ Hence $P_Ay_j\in A,$ $1\le j\le m.$ 
Put $a_j=(P_1y_j)^*,$ $1\le j\le m.$  
Then, by \eqref{Lapp-1n-1}, \eqref{Lapp-1-5},
\beq
\phi(\xi_i)=T(\xi_i)\approx_{\ep/4} TE_\lambda P_1(\xi_i)\approx_{\ep/4} L(\xi_i)=\sum_{j=1}^m a_j^*\la z_j, \xi_i\ra.
\eneq

Choose $x=\sum_{j=1}^m z_ja_j.$ Then $L(\xi_i)=\la x, \xi_i\ra,$ $i\in \N.$ Hence
\beq
\|\phi(\xi_i)-\la x, \xi_i\ra\|<\ep,\,\,\,\, i=1,2,...,k.
\eneq
It remains to show that $\|x\|\le \|\phi\|.$
However, we have 
\beq
\|x\|=\|\la x, x/\|x\|\ra\|=\|L(x/\|x\|)\|\le \|L\|\le \|\phi\|.
\eneq

\end{proof}

\begin{rem}\label{remaweakdense}
Theorem \ref{Lapp-1} also shows that 
$H$ may not be $A$-weakly closed, in the case that $H\not=H^\sharp$ (see \ref{Dwdense}). 
\end{rem}

In the following statement, note that, $H$ may  not be  a separable Banach space when $A$ is not separable.  

\begin{thm}\label{Tapproxdual}
Let $H$ be a countably generated Hilbert $A$-module and $\phi: H\to A$ be a bounded module map.
 Then,  there exists $x_n\in H,$ $n\in \N,$  such that $\|x_n\|\le \|\phi\|$ and 
 \beq
 \lim_{n\to\infty}\|\phi(\xi)-\la x_n, \xi\ra\|=0\tforal \xi\in H.
 \eneq
\end{thm}

\begin{proof}
We need to modify the proof of Theorem \ref{Lapp-1}. 
Let $\{E_n\}$ be an approximate identity for $K(H)$ (by Theorem \ref{TKHsigma}).
Let $H_1=H\oplus A,$ $P_1, $ $P_A$ and $T$ be as in the proof of Theorem \ref{Lapp-1}.
Consider $TE_nP_1\in K(H).$   As in the proof of Theorem \ref{Lapp-1},
we obtain $L_n\in K(H)$ with $\|L_n\|\le \|TE_nP_1\|\le \|\phi\|$ such that
\beq
\|TE_nP_1-L_n\|<1/2^n,
\eneq
where $L_n=\sum_{j=1}^{m(n)}\theta_{x_{n,j}, y_{n,j}}.$ 
Put $P_A(x_{n,j})=a_{n,j}^*,$ $1\le j\le m(n),$ $n\in \N.$
For any $\ep>0$ and any $z\in H,$ we may choose $n(x)$ such that, for all $n\ge n(x),$
\beq
\|E_nP_1(z)-z\|<\ep/2(\|\phi\|+1).
\eneq
Then, if $n\ge n(x),$ 
\beq
\phi(z)\approx_{\ep/2} TE_nP_1(y)\approx_{\|z\|/2^n} L_n(z)=\sum_{j=1}^{m(n)}a_{n,j}^*\la y_{n,j}, z\ra.
\eneq
Choose $x_n=\sum_{j=1}^{m(n)}y_{n,j}a_{n,j},$ $n\in \N.$ The same proof as the end of the proof 
of Theorem \ref{Lapp-1} shows that, for all $n\ge n(x),$  
\beq
\phi(z)\approx_{\ep/2+\|z\|/2^n} L(z)=\la x_n, z\ra\andeqn \|x_n\|\le \|\phi\|.
\eneq
The theorem then follows.
\end{proof}

As mentioned at the beginning of this section, injective Hilbert $A$-modules are rare. 
However, one may have some approximate extensions of bounded module maps. 
Let $B_0, B_1$ and $C$  be \CA s and $\phi: B_0\to C$ be a contractive completely  
positive linear map. If $C$ is amenable (or $B_0$ is), then, for any $\ep>0$ and any finite 
subset ${\cal F}\subset B_0,$ there exists a contractive completely positive linear 
map $\psi: B\to C$ such that
$$
\|\phi(b)-\psi(b)\|<\ep\tforal b\in {\cal F}.
$$
It is very useful  feature of amenable \CA s. With the same spirit, let us present the following
approximate extension result: 

\begin{thm}\label{Tapp-extension}
Let $A$ be a \CA, $H_0, H_1, H$ be  Hilbert $A$-modules 
such that $H_0\subset H_1$ and 
$\phi: H_0\to H$ be a bounded module map.

Then, for any $\ep>0$ and any finite subset ${\cal F}\subset H_0,$ there exists 
a bounded module map $\psi: H_1\to H$ such that $\|\psi\|\le \|\phi\|$ and
\beq
\|\psi(x)-\phi(x)\|<\ep\tforal x\in {\cal F}.
\eneq
The above may be illustrated as follows: 
                       
\begin{equation}
\begin{tikzcd}
H_1 \ar[dr, dashed, "\psi"]\\
H_0 \ar[u, hookrightarrow, "\ \, \circlearrowleft_{\varepsilon}"']  \ar[r, "\varphi"]& H
\end{tikzcd}   \hspace{0.4in} on\,\,{\cal F}.
\end{equation} 

If, in addition, $H_0$ is countably generated, then there exists  $\psi_n: H_1\to H,$ $n\in \N,$ such that
$\|\psi_n\|\le \|\phi\|$ and 
\beq
\lim_{n\to\infty}\|\psi_n(x)-\phi(x)\|=0\rforal x\in H_0.
\eneq

\end{thm}

\begin{proof}
We may assume that $\|\phi\|\not=0.$
Let  $H_2=H_0\oplus H$  and $H_3=H_1\oplus H.$
We view $H_2$ as a Hilbert $A$-submodule of $H_3.$
Let $P_0\in L(H_2)$ be the projection from $H_2$ onto $H_0,$ and 
let  $P_1, P_H\in L(H_3)$  be the projections from $H_3$ 
onto $H_1$ and onto $H,$ 
respectively.
Define $T\in B(H_2)$ by $T=\phi\circ P_0.$ 
Let ${\cal F}\subset H_0$ be a finite subset 
and $\ep>0.$

Let $\{E_\lambda\}$ be an approximate identity for $K(H_0).$ 
By Lemma \ref{Lappunit}, there exists $\lambda_0$ such that, for all $\lambda\ge \lambda_0,$ 
\beq\label{Lapp-1-1}
\|E_{\lambda}(x)-x\|<\ep/(\|\phi\|+1)\rforal x\in {\cal F}.
\eneq
Applying Theorem 1.5 of \cite{Lnbd},  we have that $B(H_2)=LM(K(H_2)).$ It follows that  (for a fixed $\lambda\ge \lambda_0$)
$TE_{\lambda}P_0\in K(H_2).$
Applying Lemma 2.13 of \cite{Lninj}, we obtain  $T_\lambda\in K(H_3)$  such that
 \beq
 \|T_\lambda\|=\|TE_\lambda P_0\|\andeqn T_\lambda|_{H_2}=TE_\lambda P_0.
\eneq
Put $\psi=P_H T_\lambda.$ Note that  (recall that $T(H_2)\subset H$)
\beq\label{Lapp-1-2}
\|\psi\|=\|TE_\lambda P_0\|\andeqn \psi |_{H_2}=TE_\lambda P_0.
\eneq
We may view $\psi$ as a bounded module map from $H_1$ to $H$ with $\|\psi\|\le \|T\|=\|\phi\|.$
We estimate that, for $x\in {\cal F},$   by \eqref{Lapp-1-1} and \eqref{Lapp-1-2}, 
\beq\nonumber
\hspace{-0.3in}\|\psi(x)-\phi(x)\| &=&\|\psi(x)-T(x)\|=\|\psi(x)-TE_\lambda P_0(x)\|+\|TE_\lambda P_0(x)-T(x)\|\\
&<& 0+\|T\|\|E_\lambda(x)-x\|<\ep.
\eneq
This proves the first part of the theorem.

For the second part of the theorem, by Theorem \ref{TKHsigma},  let $\{E_n\}$ be an approximate identity for $K(H_0)$ with 
$E_{n+1}E_n=E_n,$ $n\in \N.$   Replacing $E_\lambda$ by $E_n,$ we obtain $T_n\in K(H_3)$
such that
\beq
\|T_n\|=\|TE_nP_0\|\andeqn T_n|_{H_2}=TE_nP_0.
\eneq 
Put $\psi_n=P_H T_n,$ $n\in \N.$ Then 
\beq
\|\psi_n\|\le \|\phi\|\andeqn \psi_n|_{H_2}=TE_nP_0,\,\, n\in \N.
\eneq
We then use the fact that $\lim_{n\to\infty} \|E_n(x)-x\|=0$ for all $x\in H_0$ (see again Theorem \ref{TKHsigma}).

\end{proof}

\begin{cor}
Let $A$ be a \CA\, and $H_0\subset H$ be Hilbert $A$-modules.
Suppose that there is a bounded module map $\phi: H_0\to A.$ Then, 
for any $\ep>0$ and any finite subset ${\cal F}\subset H_0,$ there exists a 
bounded module map $\psi: H\to A$ such that
\beq
\|\phi(x)-\psi(x)\|<\ep\tforal x\in {\cal F} \tand\|\psi\|\le \|\phi\|.
\eneq

If, in addition, $H_0$ is countably  generated Hilbert $A$-module, then there exists 
a sequence of bounded module maps $\psi_n: H\to A$ such that
\beq
\lim_{n\to\infty}\|\psi_n(x)-\phi(x)\|=0\tforal x\in H \tand \|\psi_n\|\le \|\phi\| \tforal n\in \N.
\eneq
\end{cor}

\begin{rem}
In the light of Theorem \ref{Tapp-extension},  
one may think that every
Hilbert $A$-module is ``approximately injective". 
Next, let us  turn to ``approximate projectivity". 
\end{rem}

\begin{lem}\label{appprojL}
Let $A$ be a $\sigma$-unital \CA\, and $H=A^{(k)}$ (for some $k\in \N$).
Suppose that $H_1, H_2$ are Hilbert $A$-modules, $\phi: H\to H_1$ is a bounded 
module map and $s: H_2\to H_1$ is a surjective bounded module map.
Then there exists a sequence of bounded module maps 
$\psi_n: H\to H_2$   and an increasing sequence 
of Hilbert $A$-submodules $X_n\subset H$ 
such that $\overline{\cup_{n=1}^\infty X_n}=H,$ 
\beq
{s\circ \psi_{m}}|_{H_n}=\phi|_{H_n}\,\,\tforal  m\ge n.
\eneq
Moreover there exists a  sequence of module maps $T_n: H\to X_n$
such that $\|T_n\|\le 1$ and 
\beq
\lim_{n\to\infty}\|T_n(h)-h\|=0,\, T_{2n+j}|_{H_n}={\rm id}_{H_n},\,\, n,j\in \N,\\
\andeqn s\circ \phi_n(h)=\phi(T_n(h))\tforal h\in H.
\eneq
In particular,
\beq
\lim_{n\to\infty}\|s\circ \psi_n(x)-\phi(x)\|=0\tforal x\in H.
\eneq
\end{lem}

\begin{proof}
Fix a strictly positive element $e\in A.$ 
Write $A^{(n)}=\overline{e_1A}\oplus \overline{e_2A}\oplus \cdots \oplus \overline{e_nA},$
where $e_i=e^{1/2}$ ($1\le i\le n$).
Let $f_{1/n}\in C([0, \infty))$ be defined in (4) of \ref{NN1}, $n\in \N.$

For each $n\in \N,$ define $X_n=\bigoplus_{i=1}^k \overline{f_{1/n}(e_i)A}.$ 
Note that $X_n\subset X_{n+1},$ $n\in \N,$ and $\overline{\cup_{n=1}^\infty X_n}=A^{(k)}=H.$ 

%
Let $x_{n,i}=\phi(f_{1/2n}(e_i)),$ $1\le i\le k,$ and $n\in \N.$ 
Choose $y_{m,i}\in H_2$ 
 such that $s(y_{m.i})=x_{m,i},$
$1\le i\le m$   and $m\in  \N.$
Define $\psi_n: H\to H_2$ by 
\beq\label{appprojL-5}
\psi_n(h)=\sum_{i=1}^k y_{n,i}\la f_{1/2n}(e_i), h\ra\rforal h\in H.
\eneq
It follows that $\psi_n\in K(H, H_2).$ 
If $h=\bigoplus_{i=1}^k f_{1/n}(e)a_i,$  where $a_i\in A,$ then 
\beq
\psi_n(h)=\sum_{i=1}^k y_{n,i}f_{1/n}(e)a_i.
\eneq
 Hence, 
for
$h=\bigoplus_{i=1}^k f_{1/n}(e_i)a_i,$ where  $a_i\in A$ ($1\le i\le k$), and for any $j\ge 0,$ 
\beq
s(\psi_{n+j}(\sum_{i=1}^k f_{1/n}(e_i) a_i))&=&s(\sum_{i=1}^k y_{n+j,i}f_{1/n}(e)a_i)\\
&=&\sum_{i=1}^k \phi(f_{1/2(n+j)}(e_i))f_{1/n}(e)a_i\\
&=&\sum_{i=1}^k \phi(f_{1/2(n+j)}(e_i)f_{1/n}(e))a_i\\
&=&\sum_{i=1}^k \phi(f_{1/n}(e_i))a_i=\phi(\sum_{i=1}^k f_{1/n}(e_i)a_i).
\eneq
It follows that, for any $m\ge n,$ 
\beq
s\circ \psi_m|_{H_n}=\phi|_{H_n}.
\eneq

Next consider the module map $T_n: H\to X_n$  defined by
$T_n(h)=\sum_{i=1}^k f_{1/n}(e_i)\la  f_{1/n}(e_i), h\ra$
for all $h\in H.$
In other words, $T_n\in K(H, H_2)$ and $T_n(\bigoplus_{i=1}^k a_i)=\bigoplus_{i=1}^k f_{1/n}(e)^2a_i$
for any $\bigoplus_{i=1}^k a_i\in A^{(k)}.$   Hence $\|T_n\|=1.$ 
Moreover (for any $a_i\in A$),
\beq
T_{2n+j}(\bigoplus_{i=1}^k f_{1/n}(e)a_i)&=&\bigoplus_{i=1}^k f_{1/(2n+j)}(e)^2f_{1/n}(e)a_i\\
&=& \bigoplus_{i=1}^k f_{1/n}(e)a_i.
\eneq
Hence 
\beq
T_{2n+j}|_{X_n}={\rm id}_{X_n},\, \, n,j\in \N.
\eneq
It follows,  for any $h=\bigoplus_{i=1}^k a_i,$ also by \eqref{appprojL-5},  that
\beq
s\circ \psi_{n+j}(h)&=&\sum_{i=1}^k x_{n+j, i}\la f_{1/2(n+j)}(e_i), a_i\ra\\
&=&\sum_{i=1}^k \phi(f_{1/2(n+j)}(e_j))f_{1/2(n+j)}(e)a_i)=\phi(T_{n+j}(h)).
\eneq

To see the last part of the lemma, fix  $h=\bigoplus_{i=1}^k a_i\in A^{(k)}.$ Let $\ep>0.$ 
Choose $n\in \N$ such that 
\beq\label{appprojL-10}
\|f_{1/n}(e)^2a_i-a_i\|<\ep/k(\|\phi\|+1),\,\,\, 1\le i\le k.
\eneq
Hence, for any $j\in \N,$ 
\beq
\|\phi(T_{n+j}(h))-\phi(h)\|<\ep
\eneq
It follows that, for this $h$ and for any $j\in \N,$ 
\beq\nonumber
\|s\circ \psi_{n+j}(h)-\phi(h)\|
&\le &\|s\circ \psi_{n+j}(h)-\phi(T_{n+j}(h))\|+\|\phi(T_{n+j}(h))-\phi(h)\|\\\nonumber
&=&0+\ep.
\eneq
From \eqref{appprojL-10}, we also have 
\beq
\lim_{n\to\infty}\|T_n(h)-h\|=0\rforal h\in H.
\eneq
\end{proof}

\begin{thm}\label{appprojT}
Let $A$ be a $\sigma$-unital \CA\, and $H$ a countably generated Hilbert $A$-module.
Suppose that $H_1$ and $H_2$ are Hilbert $A$-modules, and 
$\phi: H\to H_1$  is a bounded module map and $s: H_2\to H_1$ is a surjective  bounded module map.
Then there exists a sequence of bounded module maps 
$\psi_n: H\to H_2$ such that
\beq
\lim_{n\to\infty}\|s\circ \psi_n(x)-\phi(x)\|=0\tforal x\in H.
\eneq
\end{thm}

The following approximately commutative diagram may illustrate  the statement of Theorem \ref{appprojT}
(with $\ep_n\to 0$):
\begin{equation}
\begin{tikzcd}
H \ar[dr, dashed, "\psi_n"]
\ar[d, 
"{\phi}_{_{_{_{ _{_{\circlearrowleft_{\varepsilon_n}}}}}}}"] \\
H_1
 &H_2 \ar[l, two heads, "s"]\,\,\,\,.
\end{tikzcd}  
\end{equation} 

\begin{proof}
We first prove the theorem for $H=H_A=l^2(A).$ 

So now we assume that $H=l^2(A).$ 
Put $Y_n=A^{(n)},$ $n\in \N.$ We  view $Y_n\subset Y_{n+1}$ as
an increasing sequence of Hilbert $A$-submodules of $H=l^2(A)$
and  the closure of the union is $H.$

By applying Lemma \ref{appprojL}, we obtain, for each $n\in \N,$ 
an increasing sequence of Hilbert $A$-submodules $X_{n,j}\subset X_{n, j+1}$
such that $\overline{\cup_{j=1}^\infty X_{n,j}}=Y_n,$ and 
there exists a sequence of bounded module maps 
$\psi_{n,j}: Y_n\to H_2$ and module maps $T_{n,j}: Y_n\to X_{n,j}$ with $\|T_{n,j}\|\le 1$ such that
\beq\label{appprojT-8}
&&{s\circ \psi_{n,j+i}}|_{X_{n,j}}=\phi|_{X_{n,j}},\,\,\,j, i\in \N, \\\label{appprojT-9}
&&T_{n, 2j+i}|_{X_{n, j}}={\rm id}|_{X_{n,j}},\,\,\, j,i\in \N,\andeqn\\\label{appprojT-10}
&&\lim_{j\to\infty}\|T_{n,j}(h)-h\|=0\rforal h\in Y_n\andeqn\\\label{appprojT-11}
&&s\circ \psi_{n,j}(h)=\phi(T_{n,j}(h))\rforal h\in Y_n\andeqn j\in \N.
\eneq
Note that $Y_n=A^{(n)}\subset A^{(n+1)}=Y_{n+1},$ $n\in \N,$ therefore we may arrange so that
\beq\label{appprojT-6}
X_{n, j}\subset X_{n+1, j}\rforal j\in \N\andeqn n\in \N.
\eneq
Define $P_n: H\to Y_n$ by $P_n(\bigoplus_{n=1}^\infty a_n)=\bigoplus_{i=1}^na_i,$
where $\sum_{n=1}^\infty a_n^*a_n$ converges, i.e, $\bigoplus_{n=1}^\infty a_n\in l^2(A)=H.$
Then $P_n$ is  a projection from $H=H_A$ onto $A^{(n)}.$

Now define 
\beq\label{appprojT-12}
\psi_n=\psi_{n,2n}\circ P_n: H\to H_2,\,\,\, n\in \N.
\eneq
We will verify that $\psi_n$ meets the requirements. 

Fix $h=\bigoplus_{n=1}^\infty a_n\in l^2(A)=H.$ 
There is $n_0\in \N$ such that, for all $n\ge m\ge n_0,$ 
\beq\label{appprojT-15}
\|P_m(h)-h\|<\ep/6(\|\phi\|+1)\andeqn \|P_n(h)-P_m(h)\|<\ep/6(\|\phi\|+1).
\eneq
For $P_{n_0}(h)\in A^{(n_0)},$ there is $n_1\in \N$ such 
that, for all $n\ge n_1$  
\beq\label{appprojT-16}
\|T_{n_0,n}(P_{n_0}(h))-P_{n_0}(h)\|<\ep/6(\|\phi\|+1).
\eneq
Since $T_{n_0, n}(P_{n_0}(h))\in X_{n_0, n}\subset X_{n,n}$ (see \eqref{appprojT-6}),
 by \eqref{appprojT-9}, 
\beq\label{appprojT-17-}
T_{n,2n}(T_{n_0, n}(P_{n_0}(h)))=T_{n_0, n}(P_{n_0}(h)).
\eneq
Hence, if $n\ge n_1+n_0,$ by \eqref{appprojT-17-} and by \eqref{appprojT-16},
\beq\label{appprojT-17}
&&\hspace{-1.4in}\|T_{n,2n}(P_{n_0}(h))-P_{n_0}(h)\|=\|T_{n, 2n}(P_{n_0}(h))-T_{n,2n}(T_{n_0, n}(P_{n_0}(h)))\|\\
&&\hspace{1in}+
\|T_{n,2n}((T_{n_0, n}(P_{n_0}(h))))-P_{n_0}(h)\|\\
&\le & \|T_{n, 2n}\|\|P_{n_0}(h)-T_{n_0, n}(P_{n_0}(h))\|+\|T_{n_0, n}(P_{n_0}(h))-P_{n_0}(h)\|\\
&<& \ep/6(\|\phi\|+1)+\ep/6(\|\phi\|+1)\\
&=&\ep/3(\|\phi\|+1).
\eneq
We also estimate that, if $n\ge n_1+n_0,$  by both parts of \eqref{appprojT-15} and the inequalities right above,
\beq\nonumber
\|T_{n,2n}(P_n(h))-h\|&\le &\|T_{n,2n}(P_n(h))-P_{n_0}(h)\|+\|P_{n_0}(h)-h\|\\\nonumber
&<&\|T_{n,2n}(P_n(h)-P_{n_0}(h))\|+\|T_{n,2n}(P_{n_0}(h))-P_{n_0}(h)\|+\ep/6(\|\phi\|+1)\\\nonumber
&<&\|T_{n,2n}\|(\ep/3(\|\phi\|+1))+\ep/3(\|\phi\|+1))+\ep/3(\|\phi\|+1)\\\label{appprojT-18}
&=&\ep/(\|\phi\|+1).
\eneq
We then  have, by \eqref{appprojT-12}, 
\eqref{appprojT-11} and \eqref{appprojT-18},
\beq\nonumber
\|s\circ \psi_n(h)-\phi(h)\|&\le&\|s\circ \psi_{n, 2n}(P_n^2(h))-\phi(T_{n,2n}(P_n(h)))\|+\|\phi(T_{n,2n}(P_n(h)))-\phi(h)\|\\\nonumber
&&<0+\ep.
\eneq

Next we consider the general case that $H$ is an arbitrary countably generated Hilbert 
$A$-module. By Kasparov's absorbing theorem (\cite{K}), there is 
a Hilbert module isomorphism $U: H\oplus l^2(A)\to l^2(A).$ 
Define $j: H\to H\oplus l^2(A)$ to be the obvious embedding and $p: H\oplus l^2(A)\to H$ 
the projection. Thus $p\circ j={\rm id}_H.$ 
Define $\Phi=\phi\circ p\circ U^{-1}: l^2(A)\to H_1.$ 
Then 
we have the following commutative diagram:
$$
\begin{array}{ccccc}
l^2(A) & \\
&&\\
_{U}\updownarrow^{U^{-1}} & \searrow^{\Phi}\\
&&\\
H\oplus l^2(A) &\stackrel{\phi\circ p}{\longrightarrow} &   H_1.
\end{array}
$$
In particular, $\Phi\circ U(x)=\phi(x)$ for all $x\in H.$
By what have been proved above, there exists a sequence of  module maps $\Psi_n: l^2(A)\to H_2$
such that
\beq
\lim_{n\to\infty} \|s\circ \Psi_n(z)-\Phi(z)\|=0\rforal z\in l^2(A).
\eneq
Define $\psi_n: H\to H_2$ by $\psi_n(x)= \Psi_n\circ U(x)$
for all $x\in H.$  It follows that
\beq,
\lim_{n\to\infty}\|s\circ\psi_n(x)-\phi(x)\|=\lim_{n\to\infty}\|s\circ \Psi_n(U(x))-\Phi(U(x))\|=0.
\eneq
\end{proof}

\begin{rem}
Note that, in Lemma \ref{appprojL}, 
for any $\ep>0,$ we may choose $\|y_n\|\le \|\tilde s^{-1}\circ \phi\|+\ep.$
So we may   estimate that $\|\psi_n\|\le \sqrt{k}(\|\tilde s^{-1}\circ \phi\|+\ep),$
which unfortunately depends on $k.$ Therefore we do not have an estimate on the norm 
of $\psi_n$ in Theorem \ref{appprojT}. This is also the reason that the proof seems a little more involved than 
expected (as we need a complicated statement in Lemma \ref{appprojL}).    If $\{\|\psi_n\|\}$ were, at least, uniformly bounded, then the proof would be much shorter.

However, under the assumption that $A$ has real rank zero, in the next statement,
we then have a better estimate on the norm 
on $\psi_n.$
\end{rem}

\begin{thm}\label{approjrr0}
Let $A$ be a $\sigma$-unital \CA\, of real rank zero and $H$ a countably generated Hilbert $A$-module.
Suppose that $\phi: H\to H_1$ is a bounded module map
and $s: H_2\to H_1$ is a surjective bounded module map.
Then there exists a sequence of module maps $\psi_n: H\to H_2$ such that
\beq
\|\psi_n\|\le \|{\tilde s}^{-1}\circ \phi\|\tand \lim_{n\to\infty}\|s\circ \psi_n(x)-\phi(x)\|=0\tforal x\in H.
\eneq
\end{thm}

\begin{proof}
Since $H$ is countably generated, 
by  Corollary 1.5 of \cite{MP} (see  
(Lemma \ref{TKHsigma}, for convenience), $K(H)$ is a $\sigma$-unital hereditary \SCA\, of $A\otimes {\cal K}$
(see also Lemma 2.13 of \cite{Lninj}). Let $\{P_m\}$ be an approximate identity for $K(H)$ consisting of projections.
Put $H_{0, m}=P_m(H).$  Then $K(H_{0,m})=P_mK(H)P_m$ is unital.
It follows from Proposition \ref{PH} that $H_{0,m}$ is algebraically finitely generated, 
say, by $\{\xi_{m, i}:  1\le i\le d(m)\},$ where $\|\xi_{m,i}\|=1.$ 
It is important to note that $H_{0, m}\subset H_{0,m'}$ if $m<m'.$ 

Fix $m\in \N.$ Let $H_{1,m }$ be the Hilbert $A$-submodule of $H_1$ 
generated by $\{x_{m,i}=\phi(\xi_{m,i}): 1\le i\le d(m)\}.$ 
Since $H_{0, m}\subset H_{0,m'},$ $H_{1,m}\subset H_{1, m'},$
if $m<m'.$

Then $K(H_{1,m})$ is a $\sigma$-unital hereditary \SCA\, of $A\otimes {\cal K}$ (see Lemma \ref{TKHsigma}), $m\in \N.$  Since 
$A$ has real rank zero, so is $K(H_{1, m}),$ $m\in \N.$ 
Let $\{p_{m, k}\}_{k=1}^\infty$ be an approximate identity for $K(H_{1, m}).$
Let $H_{1,m, k}=p_{m,k}(H_{1,m}),$ $k, m\in \N.$ 
Since $H_{1, m}\subset H_{1, m'},$ if $m\le m',$ $K(H_{1,m})\subset K(H_{1, m'})$ (see Lemma 2.13 of \cite{Lninj}).
So we may assume that 
\beq\label{PPP}
\|p_{m+1,k}p_{m,k}-p_{m,k}\|<{1\over{2^{k+m+1}(\|\phi\|+1)}},\,\,k,m\in \N.
\eneq

Note that $K(H_{1,m,k})=p_{m,k}K(H_{1,m})p_{m,k}$ which is unital.
It follows from Proposition \ref{TKHsigma} that $H_{1, m,k}$  is algebraically finitely generated, say,  
by
$z_{m,k,1}, z_{m,k,2},...,z_{m,k, r(m,k)}$ as $A$-module.  
Choose $y_{m,k,j}\in H_2$ such that $s(y_{m,k,j})=z_{m, k,j},$ $1\le j\le r(m,k),$ $k,m\in \N.$ 
Let $H_{2, m,k}$ be the Hilbert $A$-submodule of $H_2$ generated by $\{y_{m,k,j}: 1\le k=j\le r(m,k)\}$
 as Hilbert $A$-module.

Then, since $H_{1, m,k}$ is algebraically  generated by $z_{m,k,1}, z_{m,k,2},...,z_{m,k, r(i,k)}$ as $A$-module,
$s(H_{2, m,k})=H_{1, m,k}.$  In other words, $s|_{H_{2,m,k}}$ is surjective (onto $H_{1, m,k}$). 

Fix an integer $m\in \N.$ Define $\phi_{m,k}: H\to H_{1,m,k}$ by 
\beq\label{rr0-3}
\phi_{m,k}(x)=p_{m,k}\circ \phi(P_m(x))
\eneq
for all $x\in H.$ Note that $\|\phi_{m,k}\|\le \|\phi\|,$ $m,k\in \N.$ 

Now $H_{2, m,k}$ is countably generated, 
it follows from  Theorem \ref{BLT1} that there is a  module 
map  $\psi_{m,k}:  H\to H_{2, m,k}$ such that  
\beq\label{rr0-4}
&&s\circ \psi_{m,k}=\phi_{m,k}\rforal k,m\in \N,\andeqn\\
&&\|\psi_{m,k}\|\le \|\tilde s^{-1}\circ \phi_{m,k}\|\le \|\tilde s^{-1}\circ \phi\|.
\eneq
Define $\psi_n=\psi_{n,n},$ $n\in \N.$  Then $\|\psi_n\|\le \|\tilde s^{-1}\circ \phi\|$
for all $n\in \N.$ 
We will  check that 
\beq
\lim_{n\to \infty}\|s\circ \psi_n(x)-\phi(x)\|=0\rforal x\in H.
\eneq
To check this, let $\ep>0$ and $x\in H$ with $\|x\|\le 1.$  Let $M=(\|s\|+1)(\|\tilde s^{-1}\circ \phi\|+1).$
By Lemma \ref{Lappunit} , there is $n_0\in \N$ such that
\beq\label{rr0-6}
\|x-P_m(x)\|<\ep/8M\andeqn \|P_m(x)-P_{m'}(x)\|
<\ep/8M \rforal m', m\ge n_0.
\eneq
There  is  also $k_0\in \N$ with $1/k_0<\ep/4$ such that, for any $k\ge k_0,$
\beq\label{PPPP}
\|p_{n_0,k}(\phi(P_{n_0}(x)))-\phi(P_{n_0}(x))\|
<\ep/4.
\eneq
 If $n\ge n_0$ and for all $k\in \N,$
we have, if $k\ge k_0+n_0,$  by \eqref{PPPP} and \eqref{PPP}, 
\beq\label{rr0-12}
p_{k,k}(\phi(P_{n_0}(x)))&\approx_{\ep/4}&p_{k,k}p_{n_0,k_0}(\phi(P_{n_0}(x)))\\
&\approx_{1/2^{k_0}}& p_{n_0,k_0}(P_{n_0}(x))\approx_{\ep/4} P_{n_0}(x).
\eneq
Then (recall that $H_{0, n_0}\subset H_{0, n},$ if $n_0<n$), by \eqref{rr0-6}, \eqref{rr0-3}, \eqref{rr0-4}
and \eqref{rr0-12}, when $n\ge n_0+k_0+(4/\ep),$ 
\beq\nonumber
\|s\circ \psi_n(x)-\phi(x)\|&\le& \|s\circ \psi_n(x)-s\circ \psi_{n,n}(P_{n_0}(x))\|+
\|s\circ \psi_{n,n}(P_{n_0}(x))-\phi(x)\|\\\nonumber
&<&\|s\circ \psi_n\|\|x-P_{n_0}(x)\|+\|s\circ \psi_{n,n}(P_{n_0}(x))-\phi(P_{n_0}(x))\|+\\
&&\hspace{0.3in}+\|\phi(P_{n_0}(x))-\phi(x)\|\\\nonumber
&<&\ep/8+\|s\circ \psi_{n,n}(P_{n_0}(x))-\phi_{n,n}(P_{n_0}(x))\|\\
&&\hspace{0.3in}+\|p_{n, n}(\phi(P_{n_0}(x)))-\phi(P_{n_0}(x))\|+\ep/4\\\nonumber
&<&\ep/8+0+(1/2^{k_0}+\ep/2)+\ep/8<\ep.
\eneq
\end{proof}

\begin{rem}
This paper was based on a preprint \cite{LinCuntz} of 2010.   However, some  of the original part of 
\cite{LinCuntz} have been  out of dated. 
A draft of the current version was made in 2014 including improved results 
in section 5.  Parts of section 4  and section 6 are  added recently. 
\end{rem}

\vspace{0.4in}


\noindent
 email: brownl@purdue.edu

\vspace{0.2in}


\noindent email: hlin@uoregon.edu
\end{document}